\documentclass{amsart}
\usepackage[T1]{fontenc}
\usepackage[utf8]{inputenc}
\usepackage{lmodern}
\usepackage{amsmath,amssymb,amsthm,mathtools,mathrsfs}
\usepackage{microtype,needspace,etoolbox}
\usepackage[hidelinks,hypertexnames=false]{hyperref}
\numberwithin{equation}{section}
\newtheorem{theorem}{Theorem}[section]
\newtheorem{proposition}[theorem]{Proposition}
\newtheorem{lemma}[theorem]{Lemma}
\newtheorem{corollary}[theorem]{Corollary}
\theoremstyle{definition}

\theoremstyle{remark}

\BeforeBeginEnvironment{theorem}{\Needspace{10\baselineskip}}
\BeforeBeginEnvironment{proposition}{\Needspace{9\baselineskip}}
\newcommand{\Q}{\mathbb Q}
\newcommand{\HH}{\mathbb H_X}

\newcommand{\fk}[2]{f_{K_X,#2}^{(#1)}}
\newcommand{\fkk}[2]{f_{K_X^2,#2}^{(#1)}}
\newcommand{\gk}[2]{g_{K_X,#2}^{(#1)}}
\newcommand{\gkk}[2]{g_{K_X^2,#2}^{(#1)}}
\newcommand{\restr}[1]{\left.#1\right|_{H^*(X^{[n]})}}
\DeclareMathOperator{\ad}{ad}
\DeclareMathOperator{\ch}{ch}
\DeclareMathOperator{\Coeff}{Coeff}
\DeclareMathOperator{\td}{td}
\DeclareMathOperator{\Id}{Id}

\DeclareMathOperator{\sgn}{sgn}
\DeclareMathOperator{\Sym}{Sym}

\DeclareMathOperator{\Hom}{Hom}

\allowdisplaybreaks[1]
\title[The canonical class and Chern character operators]{The canonical class and Chern character operators\newline on Hilbert schemes of points}
\author{Wei-Ping Li, Zhongyi Shi}
\date{}
\subjclass[2020]{14C05, 14C17, 17B37, 17B69}
\keywords{Hilbert scheme of points, canonical class, Chern character operator, Heisenberg algebra, affine Yangian}
\hypersetup{pdftitle={The canonical class and Chern character operators on Hilbert schemes of points},pdfauthor={Wei-Ping Li, Zhongyi Shi}}
\begin{document}
\begin{abstract}
Let \(X\) be a smooth projective complex surface. We obtain explicit formulas for the Chern character operators on the Hilbert schemes of points \(X^{[n]}\) in terms of Nakajima operators, determining in particular the coefficients of the terms involving the canonical class \(K_X\) and its square. Consequently, the cohomology ring structure of the Hilbert scheme is determined. As an application, we give an alternative proof of the action of the affine Yangian \(Y_{t_1,t_2}(\widehat{\mathfrak{gl}}_1)\) on the cohomological Fock space
\[
\mathbb H_X=\bigoplus_{n\geq0}H^*(X^{[n]},\mathbb Q).
\]
\end{abstract}
\maketitle
{\small\tableofcontents}

\section{Introduction}\label{sec:introduction}
The cohomology of Hilbert schemes of points on a smooth projective complex surface carries the Heisenberg action constructed by Grojnowski and Nakajima \cite{Grojnowski,Nakajima}. A fundamental advance was made by Lehn \cite{Lehn}, who introduced the boundary operator and established its commutator formula with Nakajima operators, revealing a remarkable connection with Virasoro operators and the canonical class of the surface. His work provided an effective operator-theoretic approach to computing the Chern classes of tautological bundles, establishing a crucial link between the Heisenberg representation and the intersection theory of Hilbert schemes. 

In related developments, Lehn and Sorger
\cite{LehnSorger01} and, independently, Vasserot
\cite{VasserotHilb} described the cohomology rings
of Hilbert schemes of points on the affine plane
in terms of symmetric-group algebras.
Ellingsrud, G\"ottsche, and Lehn established universality
results for the cobordism classes of Hilbert schemes
and integrals involving Chern classes of tautological
sheaves \cite{EGL}.
Building on Lehn's operator-theoretic framework,
Li, Qin, and Wang obtained normal-ordered expansions
of higher boundary derivatives and Chern character
operators in terms of Nakajima operators \cite{LQW}. Related constructions of vertex operators on moduli spaces of sheaves were developed by Carlsson \cite{CarlssonThesis}, including contributions involving the canonical class. Carlsson and Okounkov expressed operators associated with Ext bundles on products of Hilbert schemes in terms of Nakajima operators \cite{CO}. Carlsson subsequently used vertex-operator techniques to establish quasimodularity results for equivariant Chern numbers on Hilbert schemes \cite{CarlssonQuasimodular}.
In the expansions of Li, Qin, and Wang, the coefficients of the terms independent of the canonical class and those involving the Euler class are known. In this paper, we determine the remaining coefficients involving the canonical class \(K_X\) and its square \(K_X^2\). 

Note that the cohomology ring structure of Hilbert schemes of points on surfaces is determined by the expression of Chern character operators in terms of Nakajima operators due to the results in \cite{LQW} in the projective surface case. In fact, in \cite{LehnSorger01}, Lehn and Sorger first found a complete and explicit such formulae for affine plane $\mathbb C^2$ and hence determined the cohomology ring structure in this case. On $\mathbb C^2$, the canonical class is zero. For general projective surface, Lehn \cite{Lehn} found the explicit formula for the first Chern character operator. For higher Chern character operators,  it is very difficult to find such complete formulae due to the appearance of canonical class. Tang in \cite{Tang} found the formula for the second Chern character operator. In the paper \cite{LQW}, the formula was found for the terms not involving the canonical class. Since we found the complete formula in this paper, the cohomology ring structure of the Hilbert scheme is determined. 

The explicit formulas for the canonical-class coefficients allow us to describe the corrections to the \(W\)-algebra operators \(\mathfrak J_m^p\) of Li, Qin, and Wang \cite{LQW}. Using these formulas, we give a direct proof of an affine Yangian action on the cohomological Fock space \(\mathbb H_X\). A related action was constructed by Pfammatter \cite{Pf} using a different geometric approach. More recently, Mellit, Minets, Schiffmann, and Vasserot
developed cohomological Hall algebras of zero-dimensional
sheaves on surfaces and described the actions of their
doubles on Hilbert schemes of points, establishing
connections with Heisenberg and Virasoro algebras
\cite{MMSV}. Our construction is also closely related to the rank-one \(\mathbf{SH}^{\mathbf c}\)-action of Schiffmann and Vasserot \cite{SV} on the equivariant cohomology of Hilbert schemes of points on \(\mathbb A^2\). Maulik and Okounkov constructed Yangian actions on
the equivariant cohomology of Nakajima quiver varieties
using geometric $R$-matrices \cite{MO}.
For the Jordan quiver, the rank-one Nakajima varieties
are Hilbert schemes of points on $\mathbb A^2$.
Schiffmann and Vasserot subsequently studied the
generators of preprojective cohomological Hall algebras
\cite{SVGenerators} and established an isomorphism
with the positive halves of the corresponding
Maulik--Okounkov Yangians, compatible with their
actions on quiver-variety cohomology \cite{SVYangian}. Here we work in ordinary cohomology for arbitrary smooth projective surfaces, with the contributions of \(K_X\) and \(K_X^2\) determined explicitly.

Throughout, write
\[
K_X=c_1(\omega_X),\qquad e_X=c_2(T_X),\qquad
\HH=\bigoplus_{n\geq0}H^*(X^{[n]},\Q).
\]
Here $n$ denotes the number of points, and negative Nakajima modes are creation operators. We use $\mathfrak a_m$, $\mathfrak L_m$, $\mathfrak d$, $\mathfrak G_k$, and $\mathfrak J_m^p$ for the Nakajima, Virasoro, boundary, Chern character, and $W$-algebra operators, respectively. We distinguish the class $G_k(\gamma,n)$ from its cup-product operator $\mathfrak G_k(\gamma)$.

Sections~\ref{sec:background} and~\ref{sec:Chern-background} recall the geometric setup and the relevant operator identities. In Section~\ref{sec:recursions}, we derive recursions for the coefficients involving $K_X$ and $K_X^2$ from Lehn's boundary formula. We introduce several combinatoric tools to achieve this. Iterating a merge operator and applying weighted forest identities yield explicit finite sums, including the marked-block formulas in Theorems~\ref{thm:fK} and~\ref{thm:fK2} and Proposition~\ref{prop:Phi}. The same section then reconstructs the remaining coefficients of the Chern character operators by finite shifts of Nakajima operators (Theorem~\ref{thm:reconstruction}), with an example for $\mathfrak G_2$ in Proposition~\ref{prop:G2}.

For the application, define $H^*(X)$-parametrized operators $E_k,F_k:\HH\otimes H^*(X)\to\HH$ by
\begin{equation}\label{eq:intro-EF}
E_k(\gamma)=(\ad\mathfrak d)^k\mathfrak a_{-1}(\gamma),\qquad
F_k(\gamma)=-(-\ad\mathfrak d)^k\mathfrak a_1(\gamma),\qquad k\geq0,
\end{equation}
where $E_k(\gamma)v=E_k(v\otimes\gamma)$, and similarly for $F_k$. On $X^{[n]}\times X$, consider the class-valued series
\begin{equation}\label{eq:intro-kappa}
\mathscr C_n(z)=\sum_{d\geq0}\kappa_{d,n}z^{-d-1}
=\left.\frac{c(\mathcal I_{\mathcal Z_n},z+u)/c(\mathcal I_{\mathcal Z_n},z)-1}{u}\right|_{u=K_X},
\qquad H_n(z)=1+K_X\mathscr C_n(z).
\end{equation}
Here $c$ denotes the signed Chern polynomial, and division by the independent variable $u$ is performed before specialization to $K_X$. A key identity is
\begin{equation}\label{eq:intro-goal}
\restr{[E_k(\gamma),\mathfrak a_1(1_X)]}
=\pi_{1*}(\kappa_{k,n}\cup\pi_2^*\gamma)\cup-.
\end{equation}
Its coefficient proof uses the formulas obtained in Section~\ref{sec:recursions} and the two combinatorial identities proved independently in Appendix~\ref{app:coefficient-identities}.

Section~\ref{sec:currents} begins with the main result, Theorem~\ref{thm:main}, which states the relations among the raising, lowering, and Cartan currents. The subsequent subsections prove the mixed commutator, the relations with the Cartan currents, the quadratic $EE$ and $FF$ relations, and the graded cubic Serre relations. These identities yield an $H^*(X)$-parametrized affine Yangian action on $\HH$, with the deformation parameters acting through $K_X$ and $e_X$ on the $H^*(X)$-insertions. The presentation and its verification do not require inverting either characteristic class.

\textbf{Use of generative AI}: We use ChatGPT 6 Astra during the preparation of this manuscript to assist with exploratory mathematical calculations, the organization and presentation of proofs, and English-language and LaTeX editing. The authors take full responsibility for the mathematical statements, proofs, and conclusions presented in this paper.

\section{Hilbert schemes, Heisenberg operators, and boundary differentiation}\label{sec:background}
\subsection{Hilbert schemes and the Fock space}
The Hilbert scheme $X^{[n]}$ parametrizes zero-dimensional subschemes of length $n$. It is smooth of complex dimension $2n$ (see the details in \cite{Fogarty}); set $X^{[0]}=\mathrm{pt}$. Let
\[
\mathcal Z_n\subset X^{[n]}\times X,\qquad
\pi_1:X^{[n]}\times X\to X^{[n]},\qquad \pi_2:X^{[n]}\times X\to X
\]
be the universal subscheme and projections. The homology of Hilbert schemes of points in the projective
plane was studied by Ellingsrud and Str{\o}mme \cite{ES}.
For a locally free sheaf $V$, its tautological bundle is $V^{[n]}=\pi_{1*}(\mathcal O_{\mathcal Z_n}\otimes\pi_2^*V)$, of rank $n\operatorname{rk}V$. We use rational cohomology as in \cite{Qin}.

The Poincar\'e pairings give a supersymmetric pairing on $\HH$. The vacuum is $|0\rangle=1\in H^0(X^{[0]})$. For a homogeneous insertion $\gamma$, put $\bar\gamma=\deg\gamma\bmod2$. Every bracket of inserted operators is the supercommutator $[A,B]=AB-(-1)^{|A||B|}BA$.

For $r>0$, the incidence correspondence $Q^{[n+r,n]}\subset X^{[n+r]}
\times X\times X^{[n]}$ consists of $(\xi,x,\eta)$ with $\eta\subset\xi$ and $\mathcal I_\eta/\mathcal I_\xi$ of length $r$ supported at $x$. With projections $p_+,p_X,p_-$ from $X^{[n+r]}
\times X\times X^{[n]}$ to its first, second and third factor respectively, define
\begin{align}
\mathfrak a_{-r}(\alpha)(v)&=(p_+)_*\bigl([Q^{[n+r,n]}]\cup p_X^*\alpha\cup p_-^*v\bigr),\label{eq:creation}\\
\mathfrak a_r(\alpha)(v)&=(-1)^r(p_-)_*\bigl([Q^{[n+r,n]}]\cup p_X^*\alpha\cup p_+^*v\bigr),\label{eq:annihilation}
\end{align}
using the top-dimensional fundamental cycle. Set $\mathfrak a_0=0$. By Nakajima's theorem \cite{Nakajima}, these operators
satisfy the commutation relations of Heisenberg algebra.
\begin{equation}\label{eq:Heisenberg}
[\mathfrak a_m(\alpha),\mathfrak a_n(\beta)]
=-m\delta_{m,-n}\left(\int_X\alpha\beta\right)\operatorname{Id}.
\end{equation}
The negative modes generate $\HH$ from the vacuum, while positive modes kill it eventually (\cite{Grojnowski,Nakajima,NakajimaBook}). In particular,
\begin{equation}\label{eq:Fock}
\HH\simeq\Sym\!\left(\bigoplus_{r\geq1}H^{\mathrm{even}}(X)q^r\right)
\otimes\bigwedge\!\left(\bigoplus_{r\geq1}H^{\mathrm{odd}}(X)q^r\right),
\end{equation}
with the cohomological shift $2r-2$ on the $r$th creation mode understood. The unit is
\begin{equation}\label{eq:unit}
1_{X^{[n]}}=\frac{\mathfrak a_{-1}(1_X)^n}{n!}|0\rangle.
\end{equation}
The graded adjoint, defined by $\langle Av,w\rangle=(-1)^{|A||v|}\langle v,A^\dagger w\rangle$, satisfies
\begin{equation}\label{eq:graded-adjoint}
(AB)^\dagger=(-1)^{|A||B|}B^\dagger A^\dagger,\quad
[A,B]^\dagger=-[A^\dagger,B^\dagger],\quad
\mathfrak a_{-r}(\alpha)^\dagger=(-1)^r\mathfrak a_r(\alpha).
\end{equation}

\subsection{Diagonal insertions and generalized partitions}
Let $\tau_r:X\to X^r$ be the small diagonal. Under a graded K\"unneth decomposition $\tau_{r*}\alpha=\sum_j\alpha_{j,1}\otimes\cdots\otimes\alpha_{j,r}$, write
\[
\mathfrak a_{n_1}\cdots\mathfrak a_{n_r}(\tau_{r*}\alpha)
=\sum_j\mathfrak a_{n_1}(\alpha_{j,1})\cdots\mathfrak a_{n_r}(\alpha_{j,r}).
\]
Normal ordering puts negative modes to the left. Following \cite[Definition~4.1(i)]{LQW}, a generalized partition
is a finite multiset of nonzero integers, written as
\[
\lambda=(\cdots(-2)^{m_{-2}}(-1)^{m_{-1}}1^{m_1}2^{m_2}\cdots),
\]
where $m_i=m_i(\lambda)$ denotes the multiplicity of $i$,
and only finitely many $m_i$ are nonzero.
We define
\begin{equation}\label{eq:partition-data}
\begin{aligned}
\ell(\lambda)&=\sum_{i\ne0}m_i,
&|\lambda|&=\sum_{i\ne0}i m_i,\\
s(\lambda)&=\sum_{i\ne0}i^2m_i,
&\lambda!&=\prod_{i\ne0}m_i!.
\end{aligned}
\end{equation}
The notation $\mathfrak a_\lambda(\tau_*\alpha)$ includes one small-diagonal insertion and normal ordering. The diagonal reordering identity of \cite[Lemma 3.2]{LQW} inserts $e_X$ for an additional internal contraction. 

\subsection{The boundary and Virasoro operators}
The boundary operator and its differentiation are
\begin{equation}\label{eq:boundary}
\restr{\mathfrak d}=\ch_1(\mathcal O_X^{[n]})\cup-,\qquad
A'=[\mathfrak d,A],\qquad A^{(k)}=(\ad\mathfrak d)^kA.
\end{equation}
Define $\mathfrak L_m(\alpha)=-\frac12\sum_{q\in\mathbb Z}:\mathfrak a_q\mathfrak a_{m-q}:(\tau_{2*}\alpha)$. In the convention \eqref{eq:Heisenberg}, Lehn's identities are
\begin{align}
[\mathfrak L_m(\alpha),\mathfrak a_n(\beta)]&=-n\mathfrak a_{m+n}(\alpha\beta),\label{eq:La}\\
[\mathfrak L_m(\alpha),\mathfrak L_n(\beta)]&=(m-n)\mathfrak L_{m+n}(\alpha\beta)
+\frac{m^3-m}{12}\delta_{m,-n}\!\int_Xe_X\alpha\beta\,\Id,\label{eq:LL}\\
\mathfrak a_m'(\alpha)&=m\mathfrak L_m(\alpha)-\frac{m(|m|-1)}2\mathfrak a_m(K_X\alpha).
\label{eq:Lehn}
\end{align}
See \cite{Lehn,LQW}. Equivalently,
\begin{equation}\label{eq:d-normal}
\mathfrak d=-\sum_{\substack{\ell(\lambda)=3\\|\lambda|=0}}\frac{\mathfrak a_\lambda(\tau_*1_X)}{\lambda!}
-\frac12\sum_{m\geq1}(m-1)\mathfrak a_{-m}\mathfrak a_m(\tau_{2*}K_X).
\end{equation}
The $K_X$ term in \eqref{eq:Lehn} vanishes for $m=\pm1$. Hence
\begin{equation}\label{eq:mode-shift}
[\mathfrak d,\mathfrak a_{-1}(1_X)]=-\mathfrak L_{-1}(1_X),\qquad
[-\mathfrak L_{-1}(1_X),\mathfrak a_m(\gamma)]=m\mathfrak a_{m-1}(\gamma).
\end{equation}

\section{Chern character and \texorpdfstring{$W$}{W}-algebra operators}\label{sec:Chern-background}
Set
\begin{align}
\Gamma_{k,n}&=\bigl(\ch(\mathcal O_{\mathcal Z_n})\pi_2^*\td(X)\bigr)_{k+2},\label{eq:Gamma}\\
G_k(\gamma,n)&=\pi_{1*}(\Gamma_{k,n}\cup\pi_2^*\gamma),\qquad
\restr{\mathfrak G_k(\gamma)}=G_k(\gamma,n)\cup-.
\label{eq:G-definition}
\end{align}
The subscript $k+2$ is complex cohomological degree before insertion. Put $\Gamma_{k,n}=0$ for $k<0$ and $G_k(\gamma,0)=0$. Grothendieck--Riemann--Roch gives
\begin{equation}\label{eq:G-GRR}
G_k(1_X,n)=\ch_k(\mathcal O_X^{[n]}),\quad
\restr{\mathfrak G_0(1_X)}=n\Id,\quad \mathfrak G_1(1_X)=\mathfrak d.
\end{equation}
Cup-product operators supercommute. The fundamental identity of \cite{Lehn} (or see \cite[Theorem 3.1]{LQW}) is
\begin{equation}\label{eq:G-a}
[\mathfrak G_k(\alpha),\mathfrak a_{-1}(\beta)]
=\frac1{k!}\mathfrak a_{-1}^{(k)}(\alpha\beta),\qquad k\geq0.
\end{equation}
For $k\geq2$, \cite[Theorem 4.7]{LQW} gives
\begin{align}
\mathfrak G_k(\gamma)={}&-\sum_{\substack{\ell(\lambda)=k+2\\|\lambda|=0}}
\frac{\mathfrak a_\lambda(\tau_*\gamma)}{\lambda!}
+\sum_{\substack{\ell(\lambda)=k\\|\lambda|=0}}
\frac{s(\lambda)-2}{24}\frac{\mathfrak a_\lambda(\tau_*(e_X\gamma))}{\lambda!}\notag\\
&+\sum_{\substack{\ell(\lambda)=k+1\\|\lambda|=0}}\gk{k}{\lambda}
\frac{\mathfrak a_\lambda(\tau_*(K_X\gamma))}{\lambda!}
+\sum_{\substack{\ell(\lambda)=k\\|\lambda|=0}}\gkk{k}{\lambda}
\frac{\mathfrak a_\lambda(\tau_*(K_X^2\gamma))}{\lambda!}.
\label{eq:G-expansion}
\end{align}
The numbers $g^{\bullet}_{\bullet, \bullet}$ are universal coefficients before specialization to an individual surface. No independence of $K_X^2$ and $e_X$ in its cohomology is assumed. The low degrees are
\begin{align}
\mathfrak G_0(\gamma)&=-\sum_{m\geq1}\mathfrak a_{-m}\mathfrak a_m(\tau_{2*}\gamma),\label{eq:G0}\\
\mathfrak G_1(\gamma)&=-\sum_{\substack{\ell(\lambda)=3\\|\lambda|=0}}
\frac{\mathfrak a_\lambda(\tau_*\gamma)}{\lambda!}
-\frac12\sum_{m\geq1}(m-1)\mathfrak a_{-m}\mathfrak a_m(\tau_{2*}(K_X\gamma)).\label{eq:G1}
\end{align}
For the $W$ modes, put $\mathfrak J_m^0=-\mathfrak a_m$, $\mathfrak J_m^1=\mathfrak L_m$, and, for $p\geq2$,
\begin{equation}\label{eq:J}
\begin{split}
\mathfrak J_m^p(\gamma)=p!\Bigg(&-\sum_{\substack{\ell(\lambda)=p+1\\|\lambda|=m}}
\frac{\mathfrak a_\lambda(\tau_*\gamma)}{\lambda!}\\
&+\sum_{\substack{\ell(\lambda)=p-1\\|\lambda|=m}}
\frac{s(\lambda)+m^2-2}{24}\frac{\mathfrak a_\lambda(\tau_*(e_X\gamma))}{\lambda!}\Bigg).
\end{split}
\end{equation}
Together with the identity, these span the Lie superalgebra of \cite[Theorem 5.5]{LQW}. Away from the unordered pairs $(p,q)=(0,0),(1,0),(2,0),(1,1)$ its commutator is
\begin{equation}\label{eq:W-commutator}
[\mathfrak J_m^p(\alpha),\mathfrak J_n^q(\beta)]
=(qm-pn)\mathfrak J_{m+n}^{p+q-1}(\alpha\beta)
-\frac{\Omega_{m,n}^{p,q}}{12}\mathfrak J_{m+n}^{p+q-3}(e_X\alpha\beta),
\end{equation}
where
\begin{align*}
\Omega_{m,n}^{p,q}={}&pq\bigl((p-1)(n^3-n)-(q-1)(m^3-m)\bigr)\\
&+mn\bigl(pn(p-1)(p+3q-2)-qm(q-1)(q+3p-2)\bigr).
\end{align*}
Three exceptional cases are \eqref{eq:Heisenberg}, \eqref{eq:La}, and \eqref{eq:LL}; the fourth is
\[
[\mathfrak J_m^2(\alpha),\mathfrak J_n^0(\beta)]
=-2n\mathfrak J_{m+n}^1(\alpha\beta)
+\frac{m^3-m}{6}\delta_{m,-n}\!\int_Xe_X\alpha\beta\,\Id.
\]
When $K_X\gamma=0$, the zero modes satisfy
\begin{equation}\label{eq:G-J-Kzero}
\mathfrak G_k(\gamma)=\frac{\mathfrak J_0^{k+1}(\gamma)}{(k+1)!}.
\end{equation}
For general $\gamma$, the terms involving the canonical class in \eqref{eq:G-expansion} must be retained.

\section{Explicit formulas for Chern character operators}\label{sec:recursions}
In this section, we determine the coefficients of the terms
involving the canonical class $K_X$ and its square $K_X^2$
in the Chern character operators. Our approach relies on
the higher boundary derivatives of Nakajima operators and
the fundamental commutator identity \eqref{eq:G-a}:
\begin{align*}
\mathfrak a_n^{(k)}(\alpha)
&=(\operatorname{ad}\mathfrak d)^k\mathfrak a_n(\alpha),\\
[\mathfrak G_k(\alpha),\mathfrak a_{-1}(\beta)]
&=\frac{1}{k!}\mathfrak a_{-1}^{(k)}(\alpha\beta).
\end{align*}
\subsection{Recursions for the derivative of Nakajima operators}
For $n>0$ and $k\geq2$, the higher derivative expansion has the form
\begin{align}
\mathfrak a_{-n}^{(k)}(\gamma)={}&n^k k!
\sum_{\substack{\ell(\lambda)=k+1\\|\lambda|=-n}}
\frac{\mathfrak a_\lambda(\tau_*\gamma)}{\lambda!}
-\frac{n^k k!}{24}
\sum_{\substack{\ell(\lambda)=k-1\\|\lambda|=-n}}
(s(\lambda)-1)\frac{\mathfrak a_\lambda(\tau_*(e_X\gamma))}{\lambda!}\notag\\
&+\sum_{\substack{\ell(\lambda)=k\\|\lambda|=-n}}\fk{k}{\lambda}
\frac{\mathfrak a_\lambda(\tau_*(K_X\gamma))}{\lambda!}
+\sum_{\substack{\ell(\lambda)=k-1\\|\lambda|=-n}}\fkk{k}{\lambda}
\frac{\mathfrak a_\lambda(\tau_*(K_X^2\gamma))}{\lambda!}.
\label{eq:derivative-expansion}
\end{align}
The first two terms are the universal formulas of \cite[Section 4]{LQW}; the first derivative is \eqref{eq:Lehn}. A formula for higher derivatives of Nakajima operators in equivariant cohomology was obtained in \cite[Theorem~6.4]{AlhwaimelQin}. Upon expansion, their formula is consistent with the corresponding results obtained here.

Let $I$ be a finite index set and
$\mathcal A=(p_i)_{i\in I}$ a labelled family of integers.
We call each $i\in I$ a label and $p_i$ its weight.
Distinct labels may have the same weight.
For any subset $B\subseteq I$, set
\begin{equation}\label{eq:weights}
p_B=\sum_{i\in B}p_i,\qquad
s(\mathcal A)=p_I=n>0,\qquad
\Omega(\mathcal A)=\sum_{i\in I}p_i(1-|p_i|).
\end{equation}
In particular, $s(\mathcal A)=p_I$ is called the total weight
of $\mathcal A$.
The notation $s(\mathcal A)$ is the sum of the weights, which should be distinguished from the square sum $s(\lambda)$. For $\lambda=(-p_i)_{i\in I}$, write
$f_{K_X}^{(k)}(\mathcal A)=\fk{k}{\lambda}$ and
$f_{K_X^2}^{(k)}(\mathcal A)=\fkk{k}{\lambda}$
whenever the corresponding coefficients are defined. Since the coefficients are indexed by generalized partitions,
which are multisets, these functions are symmetric in
the weights $p_i$. Equal-valued entries are still separately labelled.

Let $F$ be a function on finite multisets of integer weights.
For distinct $a,b\in I$, let $\mathcal A_{a,b}$
be obtained from $\mathcal A$ by replacing $p_a,p_b$
with $p_a+p_b$. Define the merge operator $\mathfrak U$ by
\begin{equation}\label{eq:merge-def}
(\mathfrak U F)(\mathcal A)
=\sum_{\substack{\{a,b\}\subseteq I\\a\ne b}}
(p_a+p_b)F(\mathcal A_{a,b}).
\end{equation}
Here the sum runs over unordered pairs of distinct labels.
Merging preserves the total weight of $\mathcal A$.
In the definition of $\mathfrak U$, terms with
$p_a+p_b=0$ may be omitted, since their coefficients vanish. 

For the coefficients of boundary derivatives regarding canonical classes, we have the following recursion formulas using merge operator $\mathfrak U$.
\begin{proposition}\label{prop:recursions}
The coefficients in \eqref{eq:derivative-expansion} satisfy
\begin{align}
f_{K_X}^{(k+1)}(\mathcal A)
&=\mathfrak U f_{K_X}^{(k)}(\mathcal A)
-\frac12\Omega(\mathcal A)n^k k!,
&&|I|=k+1,\quad k\geq1,\label{eq:rec-K}\\
f_{K_X^2}^{(k+1)}(\mathcal A)
&=\mathfrak U f_{K_X^2}^{(k)}(\mathcal A)
-\frac12\Omega(\mathcal A)f_{K_X}^{(k)}(\mathcal A),
&&|I|=k,\quad k\geq2,\label{eq:rec-K2}
\end{align}
with initial values
\begin{equation}\label{eq:f-initial}
f_{K_X}^{(1)}(\{n\})=\frac{n(n-1)}2,\qquad
f_{K_X^2}^{(2)}(\{n\})=\frac{n^2(n-1)^2}{4}.
\end{equation}
\end{proposition}
The proof follows the inductive argument of \cite[Theorem~4.2]{LQW}, using the derivative formula for $\mathfrak a_\lambda$ in \cite[Lemma~3.2(ii)]{LQW3} and the reordering formula in \cite[Lemma~3.4]{LQW3}. The multiplicity factors are absorbed by the normalization \(1/\lambda!\), where \(\lambda!=\prod_i m_i!\). Thus we may treat a generalized partition as a labelled family \(\mathcal A=(p_i)_{i\in I}\), without explicitly keeping track of the multiplicities.

\subsection{Combinatorial tools}\label{sec:forests}
Let $I$ be a finite index set with $|I|=m$.
For $1\le r\le m$, let $\Pi_r(I)$ denote the set of partitions of $I$ into $r$ nonempty blocks.
Each partition $\pi\in\Pi_r(I)$ is written as $\pi=\{B_1,\ldots,B_r\}$, where the blocks $B_j\subseteq I$ are pairwise disjoint and $\bigcup_{j=1}^r B_j=I$.
For example,
\[
\Pi_2(\{1,2,3\})
=
\bigl\{
\{\{1\},\{2,3\}\},
\{\{2\},\{1,3\}\},
\{\{3\},\{1,2\}\}
\bigr\}.
\]
For each $B\in\pi$, put
\[
p_B=\sum_{i\in B}p_i,
\qquad 1\le |B|\le m-r+1.
\]
Define
\begin{equation}\label{eq:forest-weight}
w(\pi)=\prod_{B\in\pi}p_B^{|B|-1},
\qquad
\mathcal A_\pi=(p_B)_{B\in\pi}.
\end{equation}
We now develop the combinatorial tools needed to solve
the recursions for the canonical-class coefficients.
We begin with a weighted-forest identity, which serves
as a key ingredient in the proof of the merge lemma.
Applying the merge lemma to the recursions in
Proposition~\ref{prop:recursions} then yields explicit
block-sum formulas for $f_{K_X}^{(k)}$ and
$f_{K_X^2}^{(k)}$.
Throughout the combinatorial identities below, zero
weights are allowed, with the convention $0^0=1$.
\begin{lemma}[Weighted forest identity]\label{lem:forest}
If $|I|=m\geq1$ and $1\leq r\leq m$, then
\begin{equation}\label{eq:forest}
\sum_{\pi\in\Pi_r(I)}w(\pi)=\binom{m-1}{r-1}n^{m-r}.
\end{equation}
This polynomial identity also holds for zero or negative weights.
\end{lemma}
\begin{proof}
The Pr\"ufer correspondence gives the weighted tree identity
$\sum_T\prod_i p_i^{\deg_T(i)-1}=(\sum_i p_i)^{m-2}$ for $m\geq2$.
Apply it to the weights $y,p_1,\ldots,p_m$ and delete the vertex of weight $y$. Each resulting component $B$ contributes $p_B^{|B|-1}$ after summing its internal tree and the attaching vertex. A singleton contributes one. Therefore
\[
(y+n)^{m-1}=\sum_{r=1}^m y^{r-1}\sum_{\pi\in\Pi_r(I)}w(\pi).
\]
Coefficient comparison proves the result. This is a polynomial argument before specialization, so it does not divide by a zero block sum. The case $m=1$ is immediate.
\end{proof}

\begin{lemma}[Merge lemma]\label{lem:merge}
For $|I|=m$ and $0\leq\ell\leq m-1$,
\begin{equation}\label{eq:merge-powers}
\mathfrak U^\ell F(\mathcal A)
=\ell!\sum_{\pi\in\Pi_{m-\ell}(I)}w(\pi)F(\mathcal A_\pi).
\end{equation}
Each block has size at most $\ell+1$.
\end{lemma}
\begin{proof}
The operator $\mathfrak U$ is linear in $F$. We prove the formula
by induction on $\ell$.

For $\ell=0$, the only partition in $\Pi_m(I)$ is the partition
into singletons, whose weight is one. For $\ell=1$, the assertion
follows directly from the definition of $\mathfrak U$.
Suppose that the formula holds for some $1\leq r\leq m-2$. Then
\[
\begin{aligned}
\mathfrak U^{r+1}F(\mathcal A)
&=r!\sum_{\pi\in\Pi_{m-r}(I)}
 w(\pi)\mathfrak U F(\mathcal A_\pi)\\
&=r!\sum_{\pi\in\Pi_{m-r}(I)}w(\pi)
 \sum_{\{\beta_1,\beta_2\}\in\binom{\pi}{2}}
 (p_{\beta_1}+p_{\beta_2})
 F\bigl((\mathcal A_\pi)_{\beta_1,\beta_2}\bigr).
\end{aligned}
\]
Here $\beta_1,\beta_2$ are labels of the weighted family
$\mathcal A_\pi$, namely blocks of $\pi$.
Thus each term $F\bigl((\mathcal A_\pi)_{\beta_1,\beta_2}\bigr)$
is of the form $F(\mathcal A_\gamma)$ for some
$\gamma\in\Pi_{m-r-1}(I)$. Write
\[
\mathfrak U^{r+1}F(\mathcal A)
=\sum_{\gamma\in\Pi_{m-r-1}(I)}
 c_\gamma F(\mathcal A_\gamma).
\]
It remains to compute $c_\gamma$.

Let $\gamma=\{C_1,\ldots,C_{m-r-1}\}$. For a block $C_i$
with $|C_i|>1$, choose $\eta\in\Pi_2(C_i)$.
Splitting $C_i$ according to $\eta$ gives
\[
\pi_{i,\eta}=(\gamma\setminus\{C_i\})\sqcup\eta
\in\Pi_{m-r}(I).
\]
Conversely, every partition $\pi\in\Pi_{m-r}(I)$ that merges to
$\gamma$ arises in this way. Therefore, by the induction hypothesis,
\[
\begin{aligned}
c_\gamma
&=r!\sum_{\substack{C\in\gamma\\|C|>1}}
 p_C\left(\prod_{D\in\gamma\setminus\{C\}}
 p_D^{|D|-1}\right)
 \sum_{\eta\in\Pi_2(C)}\prod_{B\in\eta}p_B^{|B|-1}\\
&=r!\sum_{\substack{C\in\gamma\\|C|>1}}
 (|C|-1)p_C^{|C|-1}
 \prod_{D\in\gamma\setminus\{C\}}p_D^{|D|-1}\\
&=r!\,w(\gamma)
 \sum_{\substack{C\in\gamma\\|C|>1}}(|C|-1)\\
&=(r+1)!\,w(\gamma).
\end{aligned}
\]
The second equality follows from Lemma~\ref{lem:forest}, since
\[
\sum_{\eta\in\Pi_2(C)}\prod_{B\in\eta}p_B^{|B|-1}
=(|C|-1)p_C^{|C|-2},
\]
and the last equality uses
\[
\sum_{\substack{C\in\gamma\\|C|>1}}(|C|-1)
=\sum_{C\in\gamma}(|C|-1)
=m-(m-r-1)=r+1.
\]
This proves the induction step, and hence the lemma.
\end{proof}

The merge lemma expresses iterated merges as weighted sums over set partitions. To obtain explicit formulas for the canonical-class coefficients, it remains to evaluate sums involving marked blocks and nested partitions.
The following two lemmas provide the required identities.
\begin{lemma}[A marked block]\label{lem:marked}
For $|I|=m$, $n=s(\mathcal A)$, $2\leq r\leq m$,
and a function $F$ of a block sum,
\begin{equation}\label{eq:marked}
\begin{split}
\sum_{\pi\in\Pi_r(I)}w(\pi)\sum_{B\in\pi}F(p_B)
={}&\sum_{f=1}^{m-r+1}\ \sum_{\substack{B\subseteq I\\|B|=f}}
\binom{m-f-1}{r-2}p_B^{f-1}\\
&\hspace{13mm}\cdot(n-p_B)^{m-f-r+1}F(p_B).
\end{split}
\end{equation}
For $r=1$, the left side is $n^{m-1}F(n)$.
\end{lemma}
\begin{proof}
Choose the marked subset $B$ first and partition its complement
into $r-1$ nonempty blocks. The latter weights sum by
\eqref{eq:forest}. The possibility of that partition gives precisely
$1\leq|B|\leq m-r+1$.
\end{proof}

\begin{lemma}[Refinement of a final partition]\label{lem:refinement}
For $1\leq s\leq r\leq m$ and a function $F$ of the final block sums,
\begin{equation}\label{eq:refinement}
\sum_{\pi\in\Pi_r(I)}w(\pi)
\sum_{\eta\in\Pi_s(\pi)}w(\eta;\mathcal A_\pi)F((\mathcal A_\pi)_\eta)
=\binom{m-s}{r-s}\sum_{\rho\in\Pi_s(I)}w(\rho)F(\mathcal A_\rho).
\end{equation}
Here $\eta$ partitions the labelled blocks of $\pi$ and
$w(\eta;\mathcal A_\pi)=\prod_{C\in\eta}(\sum_{B\in C}p_B)^{|C|-1}$.
\end{lemma}
\begin{proof}
Fix the final partition $\rho$. If $D\in\rho$ contains $j_D+1$ first-level blocks, then
$0\leq j_D\leq|D|-1$, $1\leq|D|\leq m-s+1$, and $\sum_Dj_D=r-s$.
Summing the two successive weights inside $D$ gives
$\binom{|D|-1}{j_D}p_D^{|D|-1}$ by \eqref{eq:forest}.
The coefficient of $t^{r-s}$ in
$\prod_{D\in\rho}(1+t)^{|D|-1}=(1+t)^{m-s}$
is $\binom{m-s}{r-s}$. This proves the claim.
\end{proof}

The empty set has one partition into zero blocks, of weight one; a nonempty set has none. We use $\binom{-1}{-1}=1$ at that endpoint and set every other invalid binomial coefficient to zero. All powers below have nonnegative exponents on their displayed ranges.

\subsection{Explicit formulas for the derivatives of Nakajima operators}
\begin{theorem}\label{thm:fK}
For $|I|=k\geq1$ and $s(\mathcal A)=n>0$,
\begin{equation}\label{eq:fK-partitions}
f_{K_X}^{(k)}(\mathcal A)
=-\frac12\sum_{r=1}^k(k-r)!(r-1)!n^{r-1}
\sum_{\pi\in\Pi_r(I)}w(\pi)\Omega(\mathcal A_\pi).
\end{equation}
For $k\geq2$, this is the explicit subset formula
\begin{align}
f_{K_X}^{(k)}(\mathcal A)={}&\frac{(k-1)!}{2}n^k(n-1)\notag\\
&-\frac12\sum_{r=2}^k(k-r)!(r-1)!n^{r-1}
\sum_{f=1}^{k-r+1}\ \sum_{\substack{B\subseteq I\\|B|=f}}
\binom{k-f-1}{r-2}\notag\\
&\hspace{30mm}\cdot p_B^f(1-|p_B|)(n-p_B)^{k-f-r+1}.
\label{eq:fK-subsets}
\end{align}
\end{theorem}
\begin{proof}
The initial value is $f_{K_X}^{(1)}(\{n\})=-\Omega(\{n\})/2$. Iterating \eqref{eq:rec-K} to this endpoint gives
\[
f_{K_X}^{(k)}(\mathcal A)=-\frac12\sum_{\ell=0}^{k-1}
\mathfrak U^\ell\bigl(\Omega(\mathcal A)n^{k-\ell-1}(k-\ell-1)!\bigr).
\]
The total remains $n$. Apply the merge lemma with $r=k-\ell$ to get \eqref{eq:fK-partitions}. The term $r=1$ is the first line of \eqref{eq:fK-subsets}. For $r\geq2$, apply \eqref{eq:marked} with $F(x)=x(1-|x|)$.
\end{proof}

\begin{corollary}\label{cor:fK-polynomial}
Equivalently, with $\sgn(0)=0$,
\begin{align}
f_{K_X}^{(k)}(\mathcal A)={}&\frac{(k-1)!}{2}n^{k+1}-\frac{k!}{2}n^k\notag\\
&+\frac{k!}2\sum_{f=1}^{k-1}
\sum_{a=f+1}^k(-1)^{a-f-1}\frac{1}{a(a-1)}\binom{k-f-1}{a-f-1}n^{k-a+1}\notag
\\
&\hspace{32mm}\cdot\sum_{\substack{B\subseteq I\\|B|=f}}\sgn(p_B)p_B^a.
\label{eq:fK-polynomial}
\end{align}
The double sum is empty when $k=1$.
\end{corollary}
\begin{proof}
In \eqref{eq:fK-partitions} write
$\Omega(\mathcal A_\pi)=n-\sum_{B\in\pi}\sgn(p_B)p_B^2$.
The constant term contributes $-k!n^k/2$ by \eqref{eq:forest}; the one-block term of the second sum is $(k-1)!n^{k+1}/2$. For a subset of size $f<k$, expand $(n-p_B)^{k-f-r+1}$ and put $j=a-f-1$. The coefficient is simplified by
\begin{align*}
&\sum_{r=2}^{k-a+2}(k-r)!(r-1)!
\binom{k-f-1}{r-2}\binom{k-f-r+1}{j}\\
&\qquad=\frac{(k-f-1)!}{j!}
\sum_{\ell=0}^{k-a}(a+\ell-2)!\frac{k-a-\ell+1}{\ell!}\\
&\qquad=(k-f-1)!(f-1)!\binom{k}{a}\binom{a-2}{f-1}\\
&\qquad=\frac{k!}{a(a-1)}\binom{k-f-1}{a-f-1}.
\end{align*}
The second equality is the twice-summed binomial identity
$\sum_{\ell=0}^M(M-\ell+1)\binom{a+\ell-2}{a-2}=\binom{a+M}{a}$.
The remaining sign and power are $(-1)^j$ and $n^{k-a+1}$, proving the formula.
\end{proof}

Now we turn to the coefficients regarding $K_X^2$. 
For $|I|=m$, $1\leq r\leq m$ and $0\leq d\leq r-1$, put
\begin{equation}\label{eq:Phi-definition}
\Phi_{r,d}(\mathcal A)=\sum_{\pi\in\Pi_r(I)}w(\pi)
\sum_{\eta\in\Pi_{r-d}(\pi)}w(\eta;\mathcal A_\pi)
\Omega(\mathcal A_\pi)\Omega((\mathcal A_\pi)_\eta).
\end{equation}
The two factors of $\Omega$ are evaluated on the first-level and final block sums.
\begin{theorem}\label{thm:fK2}
For $|I|=k-1$, $k\geq2$ and $n>0$,
\begin{equation}\label{eq:fK2-solution}
f_{K_X^2}^{(k)}(\mathcal A)=\frac14
\sum_{1\leq s\leq r\leq k-1}(k-1-r)!(r-s)!(s-1)!n^{s-1}
\Phi_{r,r-s}(\mathcal A).
\end{equation}
The summands have the finite subset expression in Proposition~\ref{prop:Phi}.
\end{theorem}
\begin{proof}
At a singleton, $f_{K_X^2}^{(2)}=-\Omega f_{K_X}^{(1)}/2$. Hence iteration of \eqref{eq:rec-K2} gives
\[
f_{K_X^2}^{(k)}(\mathcal A)=-\frac12\sum_{\ell=0}^{k-2}
\mathfrak U^\ell\bigl(\Omega(\mathcal A)f_{K_X}^{(k-1-\ell)}(\mathcal A)\bigr).
\]
Apply the merge lemma with $r=k-1-\ell$, then substitute \eqref{eq:fK-partitions} for the coefficient on the remaining $r$ blocks. The two signs give $1/4$; the two merge depths give $(k-1-r)!$ and $(r-s)!$; the remaining coefficient contributes $(s-1)!n^{s-1}$. Their weights and marked factors give \eqref{eq:Phi-definition}.
\end{proof}

\begin{proposition}\label{prop:Phi}
Let $|I|=m$, $1\leq r\leq m$, $0\leq d\leq r-1$ and $s=r-d$. Then
\begin{align}
\Phi_{r,d}(\mathcal A)={}&\sum_{f=1}^{m-s+1}\ \sum_{\substack{D\subseteq I\\|D|=f}}
\sum_{b=1}^f\ \sum_{\substack{B\subseteq D\\|B|=b}}p_B^b(1-|p_B|)\notag\\
&\cdot\sum_{j=0}^{\min(d,f-b)}\binom{m-f-s+1}{d-j}\binom{f-b-1}{j-1}
 p_D^j(p_D-p_B)^{f-b-j}\notag\\
&\cdot\Bigg[p_D(1-|p_D|)\binom{m-f-1}{s-2}(n-p_D)^{m-f-s+1}\notag\\
&\hspace{6mm}+\sum_{c=1}^{m-f-s+2}\ \sum_{\substack{C\subseteq I\backslash D\\|C|=c}}
\binom{m-f-c-1}{s-3}p_C^c(1-|p_C|)(n-p_D-p_C)^{m-f-c-s+2}\Bigg].
\label{eq:Phi-subsets}
\end{align}
The empty-partition convention of Section~\ref{sec:forests} applies. In particular, $j=0$ contributes only for $B=D$; a nonzero disjoint contribution requires $s\geq2$ and $c\leq m-f$.
\end{proposition}
\begin{proof}
Let $\rho\in\Pi_s(I)$ be the final partition determined by $(\pi,\eta)$. Inside each $D\in\rho$, let $\xi_D\in\Pi_{j_D+1}(D)$ be the first-level refinement. Then
\begin{equation}\label{eq:Phi-weight}
w(\pi)w(\eta;\mathcal A_\pi)=\prod_{D\in\rho}p_D^{j_D}w(\xi_D),\qquad
\sum_Dj_D=d,\quad0\leq j_D\leq|D|-1.
\end{equation}
Expand the first marked factor as $\sum_{B\in\pi}p_B(1-|p_B|)$, fix its block $B$, and let $D$ be the final block containing it. Set $f=|D|$, $b=|B|$, $j=j_D$. Necessarily
\[
1\leq f\leq m-s+1,\qquad1\leq b\leq f,\qquad0\leq j\leq\min(d,f-b).
\]
The remaining first-level blocks in $D$ partition $D\backslash B$ into $j$ blocks. Hence
\begin{equation}\label{eq:Phi-local}
\sum_{\substack{\xi\in\Pi_{j+1}(D)\\B\in\xi}}w(\xi)
=p_B^{b-1}\binom{f-b-1}{j-1}(p_D-p_B)^{f-b-j}.
\end{equation}
This includes $B=D,j=0$. Multiplying by $p_D^j$ and by the marked factor gives all the corresponding factors of \eqref{eq:Phi-subsets}.

Outside $D$ there are $s-1$ final blocks. The sum of their capacities $|C|-1$ is $m-f-s+1$. The binomial convolution in Lemma~\ref{lem:refinement} sums their refinements of total $d-j$ and gives $\binom{m-f-s+1}{d-j}$.

The second marked factor is a sum over final blocks. If it marks $D$, the other $s-1$ blocks contribute
$\binom{m-f-1}{s-2}(n-p_D)^{m-f-s+1}$ by the forest identity. This is the first term in the square bracket. Otherwise the marked final block is $C\subseteq I\backslash D$, with $1\leq c=|C|\leq m-f-s+2$. Its marked weight is $p_C^c(1-|p_C|)$ and the other $s-2$ blocks contribute
$\binom{m-f-c-1}{s-3}(n-p_D-p_C)^{m-f-c-s+2}$.
These are the disjoint contributions. When $s=1$, only $D=I$ occurs and there is no disjoint term; when $s=2$ and $C=I\backslash D$, the remaining partition is empty. These cases agree with the stated conventions. Summing all choices proves the formula.
\end{proof}
Derivatives of annihilation operators follow from
$\bigl(\mathfrak a_{-n}^{(k)}(\gamma)\bigr)^\dagger=(-1)^{n+k}\mathfrak a_n^{(k)}(\gamma)$.

\subsection{Formulas for $g_{K_X}$ and $g_{K_X^2}$}
We determine the last two sums of \eqref{eq:G-expansion}, using the coefficients of $\mathfrak a_{-1}^{(k)}(\gamma)$ calculated already.
For a coefficient function define
\begin{equation}\label{eq:T}
(TF)(q_1,\ldots,q_s)=\sum_{i=1}^s(q_i+1)F(q_1,\ldots,q_i+1,\ldots,q_s).
\end{equation}
The following identity deals with iterative action of $T$ on coefficient functions.
\begin{lemma}\label{lem:T-power}
For $d\geq0$,
\begin{equation}\label{eq:T-power}
(T^dF)(q_1,\ldots,q_s)=
\sum_{\substack{d_1,\ldots,d_s\geq0\\d_1+\cdots+d_s=d}}
\frac{d!}{d_1!\cdots d_s!}\prod_{i=1}^s(q_i+1)_{d_i}
F(q_1+d_1,\ldots,q_s+d_s),
\end{equation}
where $(x)_r=x(x+1)\cdots(x+r-1)$ and $(x)_0=1$. Each $d_i\leq d$.
\end{lemma}
\begin{proof}
The individual coordinate-shift operators commute. Let
\[
(T_iF)(q_1,\ldots,q_s):=(q_i+1)F(q_1,\ldots,q_i+1,\ldots,q_s),
\]
then $T_iT_j = T_jT_i$ for $i\neq j$ and $T = T_1 + \cdots + T_s$.
Their powers supply the rising factorials:
\[
(T_i^dF)(q_1,\ldots,q_s)=(q_i+1)_dF(q_1,\ldots,q_i+d,\ldots,q_s).
\]
Thus expanding $T^d$ by the multinomial theorem proves the lemma.
\end{proof}
For $\epsilon\in\{K_X,K_X^2\}$, we would like to establish a relation between $f_{\epsilon,\nu}^{(k)}$ and $g_{\epsilon,\nu\sqcup \{-|\nu|\}}^{(k)}$ to obtain explicit formulas for Chern character operators. The notation $f_{\epsilon,\,\cdot\,\sqcup\{j\}}^{(k)}$ is the coefficient function 
\[
f_{\epsilon,\,\cdot\,\sqcup\{j\}}^{(k)}(q_1,\ldots,q_s) = f_{\epsilon,(q_1,\ldots,q_s,j)}^{(k)}.
\]
For clearness, we write $F_{k,\epsilon} = f_{\epsilon,\cdot}^{(k)}$ and $F^{[j]}_{k,\epsilon} = f_{\epsilon,\,\cdot\,\sqcup\{j\}}^{(k)}$.

\begin{theorem}\label{thm:reconstruction}
Let $\lambda$ be a generalized partition with $|\lambda|=0$. Choose a positive part $n$ and set $\mu=\lambda\backslash\{n\}$. For the term involving $K_X$, take $k\geq1$, $\ell(\lambda)=k+1$; for $K_X^2$, take $k\geq2$, $\ell(\lambda)=k$. Then
\begin{equation}\label{eq:g-reconstruction}
\begin{split}
g_{\epsilon,\lambda}^{(k)}=-\frac1{k!n!}\Bigg(&(-1)^{n-1}(T^{n-1}F_{k,\epsilon})(\mu)\\
&+\sum_{j=1}^{n-1}(-1)^{n-1-j}j!
\bigl(T^{n-1-j}(F^{[j]}_{k+1,\epsilon})\bigr)(\mu)\Bigg).
\end{split}
\end{equation}
All sums are finite after applying \eqref{eq:T-power}.
\end{theorem}
\begin{proof}
Put $C_{k,n}(\gamma)=[\mathfrak G_k(\gamma),\mathfrak a_{-n}(1_X)]$.
The initial condition is $C_{k,1}=\mathfrak a_{-1}^{(k)}/k!$ by \eqref{eq:G-a}. Applying Jacobi identity to \eqref{eq:mode-shift} gives
\begin{equation}\label{eq:Ckn}
C_{k,n+1}(\gamma)=-\frac1n\left(\frac1{k!}[\mathfrak a_{-1}^{(k+1)}(\gamma),\mathfrak a_{-n}(1_X)]
+[\mathfrak a_{-1}^{(1)}(1_X),C_{k,n}(\gamma)]\right).
\end{equation}
Here $[\mathfrak G_k(\gamma),\mathfrak a_{-1}'(1_X)]=\mathfrak a_{-1}^{(k+1)}(\gamma)/k!$ follows by differentiating \eqref{eq:G-a}. The Heisenberg relation \eqref{eq:Heisenberg} gives
\begin{equation}\label{eq:remove-positive}
\left[\frac{\mathfrak a_\lambda(\tau_*(\epsilon\gamma))}{\lambda!},\mathfrak a_{-n}(1_X)\right]
=-n\frac{\mathfrak a_{\lambda\backslash\{n\}}(\tau_*(\epsilon\gamma))}{(\lambda\backslash\{n\})!}
\end{equation}
when $n$ occurs in $\lambda$; otherwise the result is zero. The multiplicity of $n$ is cancelled by $\lambda!$.

Let $c_{k,n,\epsilon}(\mu)$ be the coefficient of $\mathfrak{a}_{\mu}(\tau_*(\epsilon\gamma))/{\mu!}$ in $C_{k,n}$. Notice that $\lambda = \mu\sqcup\{n\}$. It is direct that 
\(
c_{k,1,\epsilon}=F_{k,\epsilon}/k!
\)
and $g_{\epsilon,\lambda}^{(k)}=-c_{k,n,\epsilon}(\mu)/n$. Consider a generalized partition $\nu$. Define $\nu^{(r)} = \nu\backslash \{r\}\sqcup\{r-1\}$ and $m_r(\nu)$ is the multiplicity of $r$ in $\nu$. By \eqref{eq:mode-shift}, 
\[
[\mathfrak a_{-1}'(1_X), \frac{\mathfrak a_\nu(\tau_*(\epsilon\gamma))}{\nu!}]= \sum_{\substack{
r\in\mathbb Z\backslash\{0,1\}\\
m_r(\nu)>0}}rm_{r-1}(\nu^{(r)})\frac{\mathfrak a_{\nu^{(r)}}(\tau_*(\epsilon\gamma))}{\nu^{(r)}!}.
\]
Hence
\[
\Coeff_{\mathfrak a_{\mu}/\mu!}[\mathfrak a_{-1}'(1_X),C_{k,n}(\gamma)] = \sum_{\substack{q\in\mathbb Z\\ m_q(\mu)>0}}(q+1)m_q(\mu)c_{k,n,\epsilon}(\mu\backslash\{q\}\sqcup\{q+1\}) = Tc_{k,n,\epsilon}(\mu).
\]
By \eqref{eq:Ckn}
\[
\begin{aligned}
c_{k,n+1,\epsilon}(\mu) 
&= -\frac{1}{k!n}\Coeff_{\mathfrak a_{\mu}/\mu!}[\mathfrak a_{-1}^{(k+1)}(\gamma),\mathfrak a_{-n}(1_X)] - \frac1n\Coeff_{\mathfrak a_{\mu}/\mu!}[\mathfrak a_{-1}'(1_X),C_{k,n}(\gamma)]
\\
&= \frac{1}{k!}f_{\epsilon,\mu\sqcup\{n\}}^{(k+1)} - Tc_{k,n,\epsilon}/n
=
F^{[n]}_{k+1,\epsilon}/k!-Tc_{k,n,\epsilon}/n.
\end{aligned}
\]
Solving this recursion gives the right side of \eqref{eq:g-reconstruction} multiplied by $-n$. Every nonempty zero-total partition has a positive part.

If \(n\) and \(m\) both occur in positive part of \(\lambda\), then
\[
[C_{k,n},\mathfrak a_{-m}]
=
[C_{k,m},\mathfrak a_{-n}]
\]by Jacobi identity and the commutativity of the creation operators. Comparing the coefficient of
\(\lambda\backslash\{n,m\}\) gives
\[
m\,c_{k,n,\epsilon}(\lambda\backslash\{n\})
=
n\,c_{k,m,\epsilon}(\lambda\backslash\{m\}).
\]Hence
\[
-\frac1n c_{k,n,\epsilon}(\lambda\backslash\{n\})
\]is independent of the chosen positive integer. Since the recurrence above gives its explicit finite expression, the right-hand side of \eqref{eq:g-reconstruction} is likewise independent of that choice.
\end{proof}

Below is an example of using Theorem~\ref{thm:reconstruction} to compute the coefficients of the second Chern character operator. This recovers the same formula in \cite[Proposition 5.6]{Tang}.
\subsection{The second Chern character operator}
The recursions at total weight one give
\begin{equation}\label{eq:small-f}
\fk{2}{(-(j+1),j)}=j,\quad
\fkk{3}{(-(j+1),j)}=j^2\quad(j\geq1),\quad
\fk{1}{(-1)}=\fkk{2}{(-1)}=0.
\end{equation}
Taking $\lambda=(-m,m)$ in Theorem~\ref{thm:reconstruction} yields
\begin{equation}\label{eq:small-g}
\gk{1}{(-m,m)}=-\frac1m\sum_{j=1}^{m-1}j=-\frac{m-1}{2},\qquad
\gkk{2}{(-m,m)}=-\frac1{2m}\sum_{j=1}^{m-1}j^2=-\frac{(m-1)(2m-1)}{12}.
\end{equation}
For $m>1$, $\mathfrak a_{-m}\mathfrak a_m$ commutes with $\mathfrak a_{-1}$, so the second coefficient illustrates what is recovered beyond the length-one commutator.

\begin{proposition}\label{prop:G2}
With $|\lambda^+|=\sum_{q>0}q m_q(\lambda)$, one has
\begin{align}
\mathfrak G_2(\gamma)={}&-\sum_{\substack{\ell(\lambda)=4\\|\lambda|=0}}
\frac{\mathfrak a_\lambda(\tau_*\gamma)}{\lambda!}
-\frac12\sum_{\substack{\ell(\lambda)=3\\|\lambda|=0}}(|\lambda^+|-1)
\frac{\mathfrak a_\lambda(\tau_*(K_X\gamma))}{\lambda!}\notag\\
&+\sum_{m\geq1}\frac{m^2-1}{12}\mathfrak a_{-m}\mathfrak a_m(\tau_{2*}(e_X\gamma))
-\sum_{m\geq1}\frac{(m-1)(2m-1)}{12}\mathfrak a_{-m}\mathfrak a_m(\tau_{2*}(K_X^2\gamma)).
\label{eq:G2}
\end{align}
\end{proposition}
\begin{proof}
The first and third sums are already given by \eqref{eq:G-expansion}; the last is \eqref{eq:small-g}. For the second, write $\lambda=(-a,-b,n)$ with $a,b\geq1$ and $n=a+b$. The recursions give $\fk{3}{(-u,-v,j)}=3j-1$ whenever $u,v\geq1$, $u+v=j+1$.
In \eqref{eq:g-reconstruction}, the first term is zero because $n-1$ shifts of the two negative modes must pass through zero. For $s=n-1-j$,
\[
\bigl(T^s(f_{K_X,\,\cdot\,\sqcup\{j\}}^{(3)})\bigr)(-a,-b)
=(3j-1)\sum_{d_1+d_2=s}\binom{s}{d_1}(-a+1)_{d_1}(-b+1)_{d_2}
=(3j-1)(-n+2)_s.
\]
It follows that
\[
\gk{2}{(-a,-b,n)}=-\frac{(n-2)!}{2n!}\sum_{j=1}^{n-1}j(3j-1)=-\frac{n-1}{2}.
\]
The reflected triple has the same coefficient, because $\mathfrak G_2$ is graded self-adjoint and reflection of a total-mode-zero single-diagonal monomial has no overall sign. In both cases $n=|\lambda^+|$.
\end{proof}

\section{Affine Yangian actions}\label{sec:currents}
In this section, we establish an action of the affine
Yangian $Y_{t_1,t_2}(\widehat{\mathfrak{gl}}_1)$ on the
cohomological Fock space by verifying its current relations.
A related action was constructed by Pfammatter \cite{Pf}
using a different geometric approach.

Our construction is also closely related to the deformed
$W_{1+\infty}$-algebras $W^{(\mathbf c)}(X)$ introduced by
Mellit, Minets, Schiffmann, and Vasserot \cite{MMSV}.
They identify the positive half of these algebras with the
cohomological Hall algebra of zero-dimensional sheaves
and construct a level-one action on the cohomology of
Hilbert schemes of points
\cite[Theorem~B and Section~7.3]{MMSV}.
The connection with the Li--Qin--Wang operators is
discussed in \cite[Remark~3.23]{MMSV}.

Our proof is independent of these constructions.
We use $E_k$ and $F_k$, defined through higher boundary
derivatives of Nakajima operators in the Li--Qin--Wang
convention, and express the Cartan current in terms of
the Chern polynomial of the universal ideal sheaf.
The verification relies on the explicit coefficients
obtained in Section~\ref{sec:recursions}.

\begin{theorem}\label{thm:main}\label{thm:independent-action}
Let $X$ be a smooth projective complex surface. On
$\HH=\bigoplus_{n\geq0}H^*(X^{[n]})$, consider the currents
$E_n(z)$, $F_n(w)$, $\mathscr C_n(z)$, and
$H_n(z)=1+K_X\mathscr C_n(z)$ defined by
\eqref{eq:intro-EF}--\eqref{eq:intro-kappa}.
Let $\sigma_X$ interchange the two factors of $X\times X$, and put
$K_i=\operatorname{pr}_i^*K_X$. The kernels are
\begin{align}
\zeta(s)&=1+\frac{[\Delta_X]}{s(s+K_1)},\label{eq:zeta}\\
\widetilde\zeta_\pm(s)&=\zeta(s)(s\pm K_1)(s\mp K_2).
\label{eq:tilde-zeta}
\end{align}
Products are evaluated on the indicated cohomology classes, as specified in
Subsection~\ref{subsec:current-conventions}.

For $\theta\in H^*(X\times X)$, the same-sign relations are
\begin{equation}\label{eq:main-EE}
\begin{split}
&\bigl[(E_{n+1}(z)E_n(w))
(\theta\cup\widetilde\zeta_-(w-z))\bigr]_{z<0,w<0}\\
&\qquad=\bigl[(E_{n+1}(w)E_n(z))
\bigl(\sigma_X^*(\theta\cup\widetilde\zeta_+(z-w))\bigr)
\bigr]_{z<0,w<0},\qquad n\geq0,
\end{split}
\end{equation}
\begin{equation}\label{eq:main-FF}
\begin{split}
&\bigl[(F_{n-1}(w)F_n(z))
\bigl(\sigma_X^*(\theta\cup\widetilde\zeta_-(w-z))\bigr)
\bigr]_{z<0,w<0}\\
&\qquad=\bigl[(F_{n-1}(z)F_n(w))
(\theta\cup\widetilde\zeta_+(z-w))\bigr]_{z<0,w<0},
\qquad n\geq2.
\end{split}
\end{equation}
In these two formulas, multiply by $z-w$ before projecting to
strictly negative powers of both variables. The pullback $\sigma_X^*$
acts on both the kernel and the cohomology class $\theta$.

The Cartan exchange relations are
\begin{equation}\label{eq:main-HE}
\begin{split}
&(H_{n+1}(z)E_n(w))(\theta)\\
&\quad=\left[(E_n(w)H_n(z))
\left(\sigma_X^*\left(\theta\cup
\frac{\zeta(z-w)}{\zeta(w-z)}\right)\right)
\right]_{z\gg w,\,z^{\leq0},\,w^{<0}},\qquad n\geq0,
\end{split}
\end{equation}
\begin{equation}\label{eq:main-FH}
\begin{split}
&(F_n(w)H_n(z))(\sigma_X^*\theta)\\
&\quad=\left[(H_{n-1}(z)F_n(w))
\left(\theta\cup\frac{\zeta(z-w)}{\zeta(w-z)}\right)
\right]_{z\gg w,\,z^{\leq0},\,w^{<0}},\qquad n\geq1.
\end{split}
\end{equation}
Here the kernels are expanded for $z\gg w$ before the indicated
Laurent projection. The coefficients of the Cartan series
$H_n$ and $\mathscr C_n$ commute with themselves. For homogeneous
$\alpha,\beta\in H^*(X)$, the mixed relation is
\begin{equation}\label{eq:main-EF}
\restr{[E(z;\alpha),F(w;\beta)]}
=\frac{\mathscr C_n(z;\alpha\beta)-\mathscr C_n(w;\alpha\beta)}{z-w},
\qquad n\geq0.
\end{equation}
The current $\mathscr C_n$ also satisfies \eqref{eq:divided-HE} and
its lowering counterpart. The operators $E_a$ and $F_b$ satisfy the
graded cubic Serre relations \eqref{eq:Serre}.

Together, these identities define the action of the affine Yangian
presentation specified in Subsection~\ref{sec:action}, under
$t\mapsto K_X$ and $\varepsilon\mapsto e_X$, by
\begin{equation}\label{eq:Phi-action}
\begin{aligned}
e_k&\longmapsto E_k,& f_k&\longmapsto F_k,\\
c_k&\longmapsto\bigl(v\otimes\gamma\longmapsto
\pi_{1*}(\kappa_{k,n}\cup\pi_2^*\gamma)\cup v\bigr)_{n\geq0}.
\end{aligned}
\end{equation}
The images lie in $\Hom(\HH\otimes H^*(X),\HH)$.
Products and divided commutators act by the small-diagonal product
and the reduced bracket in \eqref{eq:star-red}, respectively.
No characteristic class is inverted.
\end{theorem}

We prove the theorem in the following order. Subsection~\ref{subsec:current-conventions}
specifies the conventions for cohomological insertions and computes the Cartan coefficients.
Subsections~\ref{sec:mixed} and~\ref{sec:exchange} prove the mixed
commutator and the Cartan exchange relations. The same-sign quadratic
relations and the Serre relations are proved in
Subsections~\ref{sec:same-sign} and~\ref{sec:Serre}.
Finally, Subsection~\ref{sec:action} verifies that these identities give
the stated affine Yangian action.

\subsection{Preliminaries}\label{subsec:current-conventions}
We first specify the notation used in Theorem~\ref{thm:main} and record the
relation with Chern character operators needed in its proof.

\subsubsection{Cohomological insertions and source sectors}
We regard $E_k$ and $F_\ell$ as linear maps
\[
E_k,F_\ell:\HH\otimes H^*(X)\longrightarrow\HH,
\qquad E_k(v\otimes\gamma)=E_k(\gamma)v,\quad
F_\ell(v\otimes\gamma)=F_\ell(\gamma)v,
\]
where $E_k(\gamma)$ and $F_\ell(\gamma)$ are defined by \eqref{eq:intro-EF}. Their currents are
\[
E_n(z)=\sum_{k\geq0}E_{k,n}z^{-k-1},\qquad
F_n(w)=\sum_{\ell\geq0}F_{\ell,n}w^{-\ell-1}.
\]
For $\gamma\in H^*(X)$, we write $E_n(z;\gamma)$ and $F_n(w;\gamma)$, and omit $n$ for the corresponding maps on $\HH$. This is the same convention as for $\mathfrak a_m(\gamma)$ and $\mathfrak G_k(\gamma)$. In a fixed source sector,
\begin{equation}\label{eq:EF-sectors}
\begin{aligned}
E_{k,n}&:H^*(X^{[n]})\otimes H^*(X)\longrightarrow H^*(X^{[n+1]}),\\
F_{\ell,n}&:H^*(X^{[n]})\otimes H^*(X)\longrightarrow H^*(X^{[n-1]}),
\end{aligned}
\end{equation}
where $F_{\ell,0}=0$. For fixed $\gamma$, the operators $E_{k,n}(\gamma)$ and $F_{\ell,n}(\gamma)$ act between the indicated cohomology groups. They have parity $\bar\gamma$. Commutation with $\mathfrak d$ is understood as
$[\mathfrak d,T](v\otimes\gamma)=\mathfrak dT(v\otimes\gamma)-T(\mathfrak dv\otimes\gamma)$.

For an even class $B_n$ on $X^{[n]}\times X$ and $\gamma\in H^*(X)$, define
\begin{equation}\label{eq:left-insertion}
B_n(\gamma)v
=\pi_{1*}(B_n\cup\pi_2^*\gamma)\cup v.
\end{equation}
We use this notation for class-valued Laurent series as well, writing $B_n(z;\gamma)$ for the resulting operator. Powers, logarithms, and inverses of a class-valued series are formed \emph{before} evaluation on a cohomology class. For example, $(H_n(z)^{-1})(\gamma)$ means evaluating the formal inverse $H_n(z)^{-1}$ on $\gamma$; it is not the inverse of $H_n(z;\gamma)$.

To write products without repeatedly expanding K\"unneth decompositions, let $A,B$ denote $H^*(X)$-parametrized operators with compatible sectors, and set
\[
(AB)(\theta)=\sum_\ell A(\theta_{\ell,1})B(\theta_{\ell,2}),
\qquad \theta=\sum_\ell\theta_{\ell,1}\otimes\theta_{\ell,2}\in H^*(X\times X).
\]
Class-valued factors are evaluated using \eqref{eq:left-insertion}. Let $\sigma_X:X\times X\to X\times X$ interchange the factors. Thus
\[
\sigma_X^*(\alpha\otimes\beta)=(-1)^{\bar\alpha\bar\beta}\beta\otimes\alpha,
\qquad
[A(\alpha),B(\beta)]=(AB)(\alpha\otimes\beta)-(BA)(\sigma_X^*(\alpha\otimes\beta)).
\]
In a product such as $(\mathbf1\,E_n(w))(\theta)$, the factor $\mathbf1$ is the map $\gamma\mapsto(\int_X\gamma)\Id$. A kernel in $H^*(X\times X)$ acts by cup product on $\theta$ before evaluation. The same convention extends to any finite number of factors. In particular, reversing the order of operators requires interchanging their cohomology classes, as recorded by $\sigma_X^*$.

Brackets are formed on the direct sum and then restricted. In particular,
\begin{align}
\restr{[E_k(\alpha),F_\ell(\beta)]}
={}&E_{k,n-1}(\alpha)F_{\ell,n}(\beta)\notag\\
&-(-1)^{\bar\alpha\bar\beta}F_{\ell,n+1}(\beta)E_{k,n}(\alpha).
\label{eq:mixed-sectors}
\end{align}
For every homogeneous $\gamma$, the low modes and adjoints are
\begin{equation}\label{eq:lowest-EF}
\begin{aligned}
E_0(\gamma)&=\mathfrak a_{-1}(\gamma),&E_1(\gamma)&=-\mathfrak L_{-1}(\gamma),\\
F_0(\gamma)&=-\mathfrak a_1(\gamma),&F_1(\gamma)&=\mathfrak L_1(\gamma),\\
E_k(\gamma)^\dagger&=F_k(\gamma),&
E_{k,n}(\gamma)^\dagger&=F_{k,n+1}(\gamma).
\end{aligned}
\end{equation}
The adjoint identities follow from \eqref{eq:graded-adjoint} and the even self-adjointness of $\mathfrak d$.

\subsubsection{Relations with Chern character operators}
\begin{proposition}\label{prop:Chern-shifts}
For homogeneous $\eta,\beta$ and $r,j\geq0$,
\begin{equation}\label{eq:Chern-shifts}
[r!\mathfrak G_r(\eta),E_j(\beta)]=E_{j+r}(\eta\beta),\qquad
[r!\mathfrak G_r(\eta),F_j(\beta)]=-F_{j+r}(\eta\beta).
\end{equation}
Moreover $[\mathfrak d,E_j(\gamma)]=E_{j+1}(\gamma)$ and $[\mathfrak d,F_j(\gamma)]=-F_{j+1}(\gamma)$.
\end{proposition}
\begin{proof}
Move $(\ad\mathfrak d)^j$ through \eqref{eq:G-a} by Jacobi identity, using $[\mathfrak d,\mathfrak G_r(\eta)]=0$. This proves the first formula. Cup multiplication is graded self-adjoint, so \eqref{eq:graded-adjoint} and \eqref{eq:lowest-EF} prove the second. The last statements follow from the definitions.
\end{proof}
For $k\geq2$, the relation to the $W$-algebra modes is
\begin{align}
E_k(\gamma)={}&-\mathfrak J_{-1}^k(\gamma)
+\sum_{\substack{\ell(\nu)=k\\|\nu|=-1}}\fk{k}{\nu}\frac{\mathfrak a_\nu(\tau_*(K_X\gamma))}{\nu!}
+\sum_{\substack{\ell(\nu)=k-1\\|\nu|=-1}}\fkk{k}{\nu}\frac{\mathfrak a_\nu(\tau_*(K_X^2\gamma))}{\nu!},\label{eq:E-J}\\
F_k(\gamma)={}&\mathfrak J_1^k(\gamma)
-\sum_{\substack{\ell(\nu)=k\\|\nu|=1}}\fk{k}{-\nu}\frac{\mathfrak a_\nu(\tau_*(K_X\gamma))}{\nu!}
-\sum_{\substack{\ell(\nu)=k-1\\|\nu|=1}}\fkk{k}{-\nu}\frac{\mathfrak a_\nu(\tau_*(K_X^2\gamma))}{\nu!}.
\label{eq:F-J}
\end{align}
The Euler class term is already included in $\mathfrak J$. In the same way, \eqref{eq:G-expansion} equals $\mathfrak J_0^{k+1}/(k+1)!$ plus its two displayed terms involving the canonical class.

\subsubsection{The Cartan coefficients}
For a vector bundle with formal Chern roots $x_i$, our signed Chern polynomial is
$c(V,z)=\prod_i(z-x_i)$. Extend this multiplicatively to classes of coherent
sheaves with locally free resolutions. On $X^{[n]}\times X$, recall

\begin{equation}\label{eq:Cartan-def}
\mathscr C_n(z)=\sum_{d\geq0}\kappa_{d,n}z^{-d-1}
=\left.\frac{c(\mathcal I_{\mathcal Z_n},z+u)/c(\mathcal I_{\mathcal Z_n},z)-1}{u}\right|_{u=K_X},
\qquad H_n(z)=1+K_X\mathscr C_n(z).
\end{equation}
These are class-valued series. For $\gamma\in H^*(X)$, the associated cup-product operators are
\[
\mathscr C_n(z;\gamma)=\pi_{1*}(\mathscr C_n(z)\cup\pi_2^*\gamma)\cup-,\qquad
H_n(z;\gamma)=\pi_{1*}(H_n(z)\cup\pi_2^*\gamma)\cup-.
\]
Thus $H_n(z;\gamma)=(\int_X\gamma)\Id+\mathscr C_n(z;K_X\gamma)$. The constant class $1$ gives the family $\gamma\mapsto(\int_X\gamma)\Id$. In operator products, $H_n(z)$ and $\mathscr C_n(z)$ are evaluated by this convention; they do not denote maps with target $H^*(X^{[n]}\times X)$.

The ideal sheaf has rank one and first Chern class zero, so its signed Chern polynomial is $z+O(z^{-1})$ and has a formal Laurent inverse. The numerator vanishes at $u=0$. Formal division by $u$ is performed in a polynomial or power-series ring where multiplication by $u$ is injective, before evaluating at class $K_X$. $H_n=1+K_X\mathscr C_n$ does not recover $\mathscr C_n$ by cancellation, thus $\mathscr C_n$ contains more information than $H_n$. This is why we define two Cartan series. To see this, at $K_X=0$,
\begin{equation}\label{eq:Kzero}
H_n(z)=1,\qquad \mathscr C_n(z)=\partial_z\log c(\mathcal I_{\mathcal Z_n},z).
\end{equation}

Put $P_n(z)=\log c(\mathcal I_{\mathcal Z_n},z)$. There are
\[
\begin{gathered}
\ch(\mathcal I_{\mathcal Z_n})=1-\ch(\mathcal O_{\mathcal Z_n}),
\qquad
\ch(\mathcal O_{\mathcal Z_n})=\left(\sum_{k\ge0}\Gamma_{k,n}\right)\td(X)^{-1}
\end{gathered}
\] together with $\td(X)^{-1}=1+K_X/2+(2K_X^2-e_X)/12$. They imply
\begin{equation}\label{eq:P}
P_n(z)=\log z+\sum_{k\geq0}(k+1)!
\left(\Gamma_{k,n}+\frac{K_X}{2}\Gamma_{k-1,n}
+\frac{2K_X^2-e_X}{12}\Gamma_{k-2,n}\right)z^{-k-2}.
\end{equation}
The expansion follows from
$\log c(V,z)=\operatorname{rk}(V)\log z-\sum_{j\geq1}(j-1)!\ch_j(V)z^{-j}$.

\begin{proposition}\label{prop:kappa}
One has $\kappa_{0,n}=1$, $\kappa_{1,n}=0$, and, for $d\geq2$,
\begin{align}
\kappa_{d,n}={}&-d!\Gamma_{d-2,n}-(d-1)!K_X\Gamma_{d-3,n}
+\frac{d!}{12}(e_X-K_X^2)\Gamma_{d-4,n}\notag\\
&+\frac{K_X}{2}\sum_{\substack{i,j\geq0\\i+j=d-5}}
(i+2)!(j+2)!\,\Gamma_{i,n}\Gamma_{j,n}\notag\\
&+\frac{K_X^2}{2}\sum_{\substack{i,j\geq0\\i+j=d-6}}
(i+2)!(j+2)!\,\Gamma_{i,n}\Gamma_{j,n}\notag\\
&-\frac{K_X^2}{6}\sum_{\substack{i,j,k\geq0\\i+j+k=d-8}}
(i+2)!(j+2)!(k+2)!\,\Gamma_{i,n}\Gamma_{j,n}\Gamma_{k,n}.
\label{eq:kappa-explicit}
\end{align}
A negative upper bound in summation gives an empty sum.
\end{proposition}
\begin{proof}
Taylor expansion through $u^3$, followed by division by $u$ and specialization, is the exact identity
\begin{equation}\label{eq:Taylor}
\mathscr C_n=P_n'+\frac{K_X}{2}\bigl(P_n''+(P_n')^2\bigr)
+\frac{K_X^2}{6}\bigl(P_n'''+3P_n'P_n''+(P_n')^3\bigr).
\end{equation}
Here primes denote differentiation with respect to the corresponding variables. Higher terms vanish by $K_X^3=0$. Substituting \eqref{eq:P} into \eqref{eq:Taylor}, the linear terms give the first line of \eqref{eq:kappa-explicit}. The two leading derivatives in $K_X(P_n')^2/2$ give the next line. At $i+j=d-6$, the cross term involving the $K_X$-part of $P_n'$ contributes
$K_X^2(i+j+6)(i+2)!(j+2)!\Gamma_{i,n}\Gamma_{j,n}/4$.
It cancels the contribution of $K_X^2P_n'P_n''/2$. The factor $1/z$ in the remaining cubic term gives the quadratic $K_X^2$ line. The final line arises from the cubic product of the terms involving \(\Gamma_{k,n}\) in \(P_n'(z)\), each of which carries a minus sign.

Since $P_n(z)=\log z+O(z^{-2})$, we have
$\mathscr C_n(z)=z^{-1}+O(z^{-3})$.
Thus $\kappa_{0,n}=1$ and $\kappa_{1,n}=0$.
\end{proof}

\begin{proposition}\label{prop:logH}
The logarithm of the undivided version of current $H_n(z)$ is
\begin{equation}\label{eq:logH}
\log H_n(z)=\log(1+K_X/z)-K_X\sum_{k\geq0}(k+2)!\,\Gamma_{k,n}z^{-k-3}.
\end{equation}
\end{proposition}
\begin{proof}
Use $\log H_n=P_n(z+K_X)-P_n(z)=K_XP_n'+K_X^2P_n''/2$. The derivative of the $K_X\Gamma_{k-1,n}/2$ term of \eqref{eq:P} is
$-K_X^2\sum_{k\geq0}(k+3)!\Gamma_{k,n}z^{-k-4}/2$,
which cancels the second derivative of the leading $\Gamma$ term. The remaining characteristic-class terms vanish for degree reasons. The term $\log z$ in $P_n(z)$ contributes $K_X/z-K_X^2/(2z^2)$.
\end{proof}

Since every coefficient of $H_n(z;\alpha)$ and $\mathscr C_n(w;\beta)$
acts by cup product, these coefficients supercommute. This proves the
Cartan commutativity assertion of Theorem~\ref{thm:main}.

\subsection{The \(E\)-\(F\) commutator}\label{sec:mixed}
We prove \eqref{eq:main-EF} by first identifying the commutator with
$F_0=-\mathfrak a_1$ and then obtaining all other lowering modes by
boundary differentiation. The coefficient identities used in the first
step are proved independently in Appendix~\ref{subsec:proof-two-identities}.

Let $\lambda=(q_i)_{i\in I}$ with $q_i\ne0$ and $\sum_iq_i=0$, and set $\lambda_B=(q_i)_{i\in B}$. In the two- and three-block sums below the conditions are
\begin{equation}\label{eq:zero-block-bounds}
|\lambda_B|=0,\qquad 2\leq|B|\leq|I|-2(r-1)\quad(B\in\pi),\qquad r=2,3.
\end{equation}
\begin{proposition}\label{prop:two-identities}
For $\ell(\lambda)=d-1$, $d\geq3$,
\begin{equation}\label{eq:identity-K}
\fk{d}{\lambda\sqcup\{-1\}}=-d!\gk{d-2}{\lambda}+(d-1)!
+\sum_{\substack{\pi\in\Pi_2(I)\\\text{conditions }\eqref{eq:zero-block-bounds}}}\prod_{B\in\pi}|B|!.
\end{equation}
For $\ell(\lambda)=d-2$, $d\geq4$,
\begin{align}
\fkk{d}{\lambda\sqcup\{-1\}}={}&-d!\gkk{d-2}{\lambda}-(d-1)!\gk{d-3}{\lambda}+\frac{d!}{12}\notag\\
&+\sum_{\substack{\pi\in\Pi_2(I)\\\text{conditions }\eqref{eq:zero-block-bounds}}}\prod_{B\in\pi}|B|!\notag\\
&-\sum_{\substack{\pi\in\Pi_2(I)\\\text{conditions }\eqref{eq:zero-block-bounds}}}
\sum_{C\in\pi}(|C|+1)!\gk{|C|-1}{\lambda_C}\prod_{B\in\pi\backslash\{C\}}|B|!\notag\\
&+\sum_{\substack{\pi\in\Pi_3(I)\\\text{conditions }\eqref{eq:zero-block-bounds}}}\prod_{B\in\pi}|B|!.
\label{eq:identity-K2}
\end{align}
\end{proposition}
Their proof is given in Appendix~\ref{subsec:proof-two-identities}. 

\begin{lemma}\label{lem:products}
For $r\geq1$, let $\tau_{r*}\eta=\sum_a\eta_{a,1}\otimes\cdots\otimes\eta_{a,r}$ in the graded K\"unneth convention. Then, for $i_1,\ldots,i_r\geq0$,
\begin{equation}\label{eq:Gamma-products}
\pi_{1*}(\Gamma_{i_1,n}\cdots\Gamma_{i_r,n}\cup\pi_2^*\eta)
=\sum_a\prod_{j=1}^rG_{i_j}(\eta_{a,j},n).
\end{equation}
The corresponding identity of multiplication operators follows by cup product.
\end{lemma}
\begin{proof}
Let $p:X^{[n]}\times X^r\to X^{[n]}$ and
$q_j:X^{[n]}\times X^r\to X^{[n]}\times X$ be the projections. Base change and the projection formula give
\[
\prod_{j=1}^rG_{i_j}(\eta_{a,j},n)
=p_*\left(\prod_{j=1}^r q_j^*(\Gamma_{i_j,n}\cup\pi_2^*\eta_{a,j})\right).
\]
After summing over $a$, apply the projection formula to $\operatorname{id}_{X^{[n]}}\times\tau_r$:
\begin{align*}
\sum_a\prod_{j=1}^rG_{i_j}(\eta_{a,j},n)
&=p_*\left(\prod_{j=1}^rq_j^*\Gamma_{i_j,n}\cup
(\operatorname{id}_{X^{[n]}}\times\tau_r)_*\pi_2^*\eta\right)\\
&=\pi_{1*}\left(\Gamma_{i_1,n}\cdots\Gamma_{i_r,n}\cup\pi_2^*\eta\right).
\end{align*}
Here $(\operatorname{id}\times\tau_r)^*q_j^*\Gamma_{i_j,n}=\Gamma_{i_j,n}$ for every $j$.
\end{proof}

\begin{theorem}\label{thm:Cartan-comparison}
For all $n,d\geq0$ and homogeneous $\gamma$,
\begin{equation}\label{eq:Cartan-comparison}
\restr{[E_d(\gamma),\mathfrak a_1(1_X)]}
=\pi_{1*}(\kappa_{d,n}\cup\pi_2^*\gamma)\cup-.
\end{equation}
\end{theorem}
\begin{proof}
Write $\mathcal C_d(\gamma)=[E_d(\gamma),\mathfrak a_1(1_X)]$.
We compare the two sides of \eqref{eq:Cartan-comparison} through their
normal-ordered expansions. All operators below are restricted to
$H^*(X^{[n]})$.

\smallskip
\noindent\emph{(1) Expansion of the left-hand side.}
Commuting a divided monomial with $\mathfrak a_1(1_X)$ removes one
occurrence of $-1$ with coefficient one. Applying this to the expansion
of $E_d(\gamma)=\mathfrak a_{-1}^{(d)}(\gamma)$ in
\eqref{eq:derivative-expansion} gives, for $d\geq2$,
\begin{align}
\mathcal C_d(\gamma)
={}&d!\sum_{\substack{\ell(\lambda)=d\\|\lambda|=0}}
\frac{\mathfrak a_\lambda(\tau_*\gamma)}{\lambda!}
-\frac{d!}{24}\sum_{\substack{\ell(\lambda)=d-2\\|\lambda|=0}}
s(\lambda)\frac{\mathfrak a_\lambda(\tau_*(e_X\gamma))}{\lambda!}
\notag\\
&+\sum_{\substack{\ell(\lambda)=d-1\\|\lambda|=0}}
\fk{d}{\lambda\sqcup\{-1\}}
\frac{\mathfrak a_\lambda(\tau_*(K_X\gamma))}{\lambda!}
\notag\\
&+\sum_{\substack{\ell(\lambda)=d-2\\|\lambda|=0}}
\fkk{d}{\lambda\sqcup\{-1\}}
\frac{\mathfrak a_\lambda(\tau_*(K_X^2\gamma))}{\lambda!}.
\label{eq:C-expansion}
\end{align}
Here the Euler coefficient uses
$s(\lambda\sqcup\{-1\})-1=s(\lambda)$.

Assume first that $d\geq4$. Substitute
\eqref{eq:identity-K}--\eqref{eq:identity-K2} from
Proposition~\ref{prop:two-identities} into \eqref{eq:C-expansion}.
By \eqref{eq:G-expansion}, with the low-degree formulas
\eqref{eq:G0}--\eqref{eq:G1}, the terms without block sums combine into
\begin{equation}\label{eq:C-linear}
-d!\mathfrak G_{d-2}(\gamma)
-(d-1)!\mathfrak G_{d-3}(K_X\gamma)
+\frac{d!}{12}\mathfrak G_{d-4}((e_X-K_X^2)\gamma).
\end{equation}
Expression~\eqref{eq:C-linear} is also exactly the contribution of the
first line of \eqref{eq:kappa-explicit} to the right-hand side of
\eqref{eq:Cartan-comparison}. We therefore subtract this common
expression from both sides and compare the remaining terms.

For each $\lambda=(q_i)_{i\in I}$ below, put $m=\ell(\lambda)=|I|$.
The notation \[\sum^*_{\pi\in\Pi_r(I)}\] means that every block $B\in\pi$ satisfies
\[
|\lambda_B|=0,\qquad
2\leq |B|\leq m-2(r-1),\qquad r=2,3,
\]
as in \eqref{eq:zero-block-bounds}. Sums over $\Pi_r(I)$ are unordered;
sums written as $B_1\sqcup\cdots\sqcup B_r=I$ are over ordered
block decompositions.

After subtracting \eqref{eq:C-linear}, the left-hand side is
\begin{equation}\label{eq:Cartan-left-blocks}
\begin{aligned}
&\sum_{\substack{\ell(\lambda)=d-1\\|\lambda|=0}}
\frac{\mathfrak a_\lambda(\tau_*(K_X\gamma))}{\lambda!}
\sum_{\pi\in\Pi_2(I)}^{*}\prod_{B\in\pi}|B|!
\\
&\quad+
\sum_{\substack{\ell(\lambda)=d-2\\|\lambda|=0}}
\frac{\mathfrak a_\lambda(\tau_*(K_X^2\gamma))}{\lambda!}
\Bigg[
\sum_{\pi\in\Pi_2(I)}^{*}\prod_{B\in\pi}|B|!
\\
&\qquad\qquad-
\sum_{\pi\in\Pi_2(I)}^{*}\sum_{C\in\pi}
(|C|+1)!\gk{|C|-1}{\lambda_C}
\prod_{B\in\pi\backslash\{C\}}|B|!
+\sum_{\pi\in\Pi_3(I)}^{*}\prod_{B\in\pi}|B|!
\Bigg].
\end{aligned}
\end{equation}
Thus it remains to recover these block sums from the last three
lines of Proposition~\ref{prop:kappa}.

\smallskip
\noindent\emph{(2) Expansion of the right-hand side.}
For $\tau_{r*}\eta=\sum_a\eta_{a,1}\otimes\cdots\otimes\eta_{a,r}$,
Lemma~\ref{lem:products} gives
\[
\pi_{1*}\left(
\Gamma_{i_1,n}\cdots\Gamma_{i_r,n}\cup\pi_2^*\eta
\right)\cup-
=\sum_a\mathfrak G_{i_1}(\eta_{a,1})\cdots
\mathfrak G_{i_r}(\eta_{a,r}).
\]
Apply this identity to the quadratic and cubic terms of
\eqref{eq:kappa-explicit}, and expand the operators using
\eqref{eq:G-expansion} and \eqref{eq:G0}--\eqref{eq:G1}.
Commuting Nakajima operators produces terms involving $e_X$,
which vanish here since $K_Xe_X=0$.
Thus, in the terms involving $K_X$, we only need to consider
the product of two leading terms or the product of a leading
term and a term involving $K_X$.
In the terms involving $K_X^2$, only products of leading
terms contribute.

A leading term of $\mathfrak G_i$ has length $i+2$ and coefficient
$-1$, whereas its $K_X$-term has length $i+1$ and coefficient
$g_{K_X,\mu}^{(i)}$. Consequently the factors $(i+2)!$ in
\eqref{eq:kappa-explicit} give
\[
\begin{aligned}
(i+2)!(-1)&=-|B|!,
&& |B|=i+2,\\
(i+2)!\gk{i}{\lambda_B}
&=(|B|+1)!\gk{|B|-1}{\lambda_B},
&& |B|=i+1.
\end{aligned}
\]
Suppose $\lambda=\mu_1\sqcup\cdots\sqcup\mu_r$.
When the product of Nakajima operators is written
in the form $\mathfrak a_\lambda/\lambda!$,
the coefficient is $\lambda!/(\mu_1!\cdots\mu_r!)$.
This counts the ways of assigning the labelled entries
of $\lambda$ to $r$ ordered blocks with the prescribed
multisets $\mu_1,\ldots,\mu_r$.
Thus we obtain sums over ordered block decompositions of $I$.

The term involving $K_X/2$ gives monomials of length $d-1$
when both Chern character operators contribute their leading
terms, and monomials of length $d-2$ when one contributes
its leading term and the other its $K_X$-term.
The latter contribution has a minus sign.
The terms involving $K_X^2/2$ and $-K_X^2/6$
use only the leading terms and both give monomials
of length $d-2$.
The cubic contribution is positive since $-(-1)^3=1$.
Thus, after subtracting \eqref{eq:C-linear}, the right-hand side is
\begin{equation}\label{eq:Cartan-right-blocks}
\begin{aligned}
&\sum_{\substack{\ell(\lambda)=d-1\\|\lambda|=0}}
\frac{\mathfrak a_\lambda(\tau_*(K_X\gamma))}{\lambda!}
\left[\frac12\sum_{B\sqcup C=I}^{*}|B|!\,|C|!\right]
\\
&\quad+
\sum_{\substack{\ell(\lambda)=d-2\\|\lambda|=0}}
\frac{\mathfrak a_\lambda(\tau_*(K_X^2\gamma))}{\lambda!}
\Bigg[\frac12\sum_{B\sqcup C=I}^{*}|B|!\,|C|!
\\
&\qquad\qquad-
\frac12\sum_{B\sqcup C=I}^{*}
\Bigl(
(|B|+1)!\,|C|!\,\gk{|B|-1}{\lambda_B}
+|B|!\,(|C|+1)!\,\gk{|C|-1}{\lambda_C}
\Bigr)
\\
&\qquad\qquad+
\frac16\sum_{B_1\sqcup B_2\sqcup B_3=I}^{*}
|B_1|!\,|B_2|!\,|B_3|!
\Bigg].
\end{aligned}
\end{equation}
The first sum and the negative marked-block sum come from the
$K_X/2$ term. The other two sums come from the $K_X^2/2$ and
$-K_X^2/6$ terms, respectively. Note that these sums are over ordered partitions of $I$.

\smallskip
\Needspace{18\baselineskip}
\noindent\emph{(3) Comparison of the two expansions.}
To identify \eqref{eq:Cartan-left-blocks} and
\eqref{eq:Cartan-right-blocks}, we use the following three identities:
\begin{equation}\label{eq:Cartan-two-blocks}
\frac12\sum_{B\sqcup C=I}^{*}|B|!\,|C|!
=\sum_{\pi\in\Pi_2(I)}^{*}\prod_{B\in\pi}|B|!,
\end{equation}
\begin{equation}\label{eq:Cartan-marked-blocks}
\begin{aligned}
&\frac12\sum_{B\sqcup C=I}^{*}
\Bigl(
(|B|+1)!\,|C|!\,\gk{|B|-1}{\lambda_B}
+|B|!\,(|C|+1)!\,\gk{|C|-1}{\lambda_C}
\Bigr)
\\
&\qquad=
\sum_{\pi\in\Pi_2(I)}^{*}\sum_{C\in\pi}
(|C|+1)!\gk{|C|-1}{\lambda_C}
\prod_{B\in\pi\backslash\{C\}}|B|!,
\end{aligned}
\end{equation}
\begin{equation}\label{eq:Cartan-three-blocks}
\frac16\sum_{B_1\sqcup B_2\sqcup B_3=I}^{*}
|B_1|!\,|B_2|!\,|B_3|!
=\sum_{\pi\in\Pi_3(I)}^{*}\prod_{B\in\pi}|B|!.
\end{equation}
Indeed, each unordered two-block partition has two orderings, and
each unordered three-block partition has six. For
\eqref{eq:Cartan-marked-blocks}, each fixed marked partition also
occurs twice, once for each position of the marked block. These
counts remain valid when distinct blocks have the same sizes or
the same multisets of parts.

Equation~\eqref{eq:Cartan-two-blocks} matches both unmarked
two-block sums, while \eqref{eq:Cartan-marked-blocks} and
\eqref{eq:Cartan-three-blocks} match the marked and three-block sums.
The remaining terms on the two sides are therefore equal. Together
with their common contribution \eqref{eq:C-linear}, this proves
\eqref{eq:Cartan-comparison} for $d\geq4$.

\smallskip
\noindent\emph{(4) Low-degree cases.}
For $d=0,1$, \eqref{eq:Heisenberg}, \eqref{eq:La}, and
\eqref{eq:lowest-EF} give
\[
\mathcal C_0(\gamma)=\left(\int_X\gamma\right)\Id,
\qquad \mathcal C_1(\gamma)=0.
\]
For $d=2,3$, the block sums are empty. Equations
\eqref{eq:C-expansion}, \eqref{eq:identity-K}, and
\eqref{eq:G0}--\eqref{eq:G1} give
\[
\begin{aligned}
\mathcal C_2(\gamma)&=-2\mathfrak G_0(\gamma),\\
\mathcal C_3(\gamma)&=-6\mathfrak G_1(\gamma)
-2\mathfrak G_0(K_X\gamma).
\end{aligned}
\]
The possible empty correction for $d=2$ is zero by
\eqref{eq:f-initial} at total weight one. These expressions agree
with the operators associated with
\[
\kappa_{0,n}=1,\qquad \kappa_{1,n}=0,\qquad
\kappa_{2,n}=-2\Gamma_{0,n},\qquad
\kappa_{3,n}=-6\Gamma_{1,n}-2K_X\Gamma_{0,n},
\]
respectively, by Proposition~\ref{prop:kappa}. This completes the proof.
\end{proof}

\begin{proposition}\label{prop:all-mixed}
For homogeneous $\alpha,\beta$ and $a,b,n\geq0$,
\begin{equation}\label{eq:all-mixed}
\restr{[E_a(\alpha),F_b(\beta)]}
=-\pi_{1*}(\kappa_{a+b,n}\cup\pi_2^*(\alpha\beta))\cup-.
\end{equation}
\end{proposition}
\begin{proof}
Commuting with $F_0(\beta)=-\mathfrak a_1(\beta)$ removes a $-1$ mode and gives
$[E_a(\alpha),F_0(\beta)]=-\mathcal C_a(\alpha\beta)$, with the cohomology classes in the indicated order. By Theorem~\ref{thm:Cartan-comparison}, $\mathcal C_a$ is cup multiplication, so it commutes with $\mathfrak d$. If the assertion holds at $b$, Jacobi identity  gives
\[
0=[\mathfrak d,[E_a(\alpha),F_b(\beta)]]
=[E_{a+1}(\alpha),F_b(\beta)]-[E_a(\alpha),F_{b+1}(\beta)],
\]
proving the induction. Since $\mathfrak d$ is even, no new sign occurs.
\end{proof}
Taking the generating series of \eqref{eq:all-mixed}
and using the identity
\begin{equation}\label{eq:Laurent-difference}
\frac{\mathscr C_n(z)-\mathscr C_n(w)}{z-w}
=-\sum_{a,b\geq0}
\kappa_{a+b,n}z^{-a-1}w^{-b-1},
\end{equation}
we obtain \eqref{eq:main-EF}.

\subsection{Relations with the Cartan current}\label{sec:exchange}
We first prove \eqref{eq:main-HE}--\eqref{eq:main-FH}
using Propositions~\ref{prop:Chern-shifts} and
\ref{prop:logH}. We then prove the corresponding
relations for $\mathscr C_n$ without cancelling $K_X$.

\subsubsection{The kernel and the logarithmic relation}
Recall \eqref{eq:zeta}, and put $e_i=\operatorname{pr}_i^*e_X$ on $X\times X$.
We expand the inverses of these kernels as formal Laurent series in \(s^{-1}\).
The choice of $K_1$ or $K_2$ in a term already multiplied by $[\Delta_X]$
makes no difference.

\begin{lemma}\label{lem:zeta}
In ordinary cohomology of $X\times X$,
\begin{equation}\label{eq:zeta-ratio}
\frac{\zeta(s)}{\zeta(-s)}=1-\frac{2\Delta_{X*}K_X}{s^3},\qquad
(\Delta_{X*}K_X)^2=0.
\end{equation}
\end{lemma}
\begin{proof}
The projection formula and diagonal self-intersection give
\begin{equation}\label{eq:diagonal-selfintersection}
[\Delta_X]K_1=[\Delta_X]K_2=\Delta_{X*}K_X,\qquad
[\Delta_X]^2=\Delta_{X*}e_X.
\end{equation}
Consequently $[\Delta_X]^2K_1=\Delta_{X*}(e_XK_X)=0$ and
$[\Delta_X]^3=\Delta_{X*}(e_X^2)=0$.
Since $K_1^3=0$, the two kernels are exactly
\begin{equation}\label{eq:zeta-expansions}
\zeta(\pm s)=1+\frac{[\Delta_X]}{s^2}
\mp\frac{[\Delta_X]K_1}{s^3}
+\frac{[\Delta_X]K_1^2}{s^4}.
\end{equation}
It follows that
\[
\frac{\zeta(s)}{\zeta(-s)}
=1-\frac{2[\Delta_X]K_1}{s^3}\,\zeta(-s)^{-1}.
\]
The inverse exists because its constant coefficient is one. Every nonconstant term of that inverse contains $[\Delta_X]$; its product with $[\Delta_X]K_1$ is zero by \eqref{eq:diagonal-selfintersection}. Thus this expression is $1-2\Delta_{X*}K_X/s^3$ without any truncation error. Finally,
\[
(\Delta_{X*}K_X)^2
=[\Delta_X]^2K_1^2=\Delta_{X*}(e_XK_X^2)=0.
\]
This proves both assertions. In particular, the argument does not assert $[\Delta_X]^2=0$.
\end{proof}

The symbol $[\cdot]_{z\gg w,\,z^{\leq0},\,w^{<0}}$ means expansion in $z^{-1}$ with polynomial coefficients in $w$, followed by the indicated Laurent projection. Every prescribed coefficient receives finitely many contributions.

\begin{theorem}\label{thm:exchange}
The Cartan exchange identities \eqref{eq:main-HE} and \eqref{eq:main-FH} hold for
$n\geq0$ and $n\geq1$, respectively. For $\theta=\alpha\otimes\beta$,
the class $\alpha$ is paired with the Cartan current and $\beta$ with the raising
or lowering current; $\sigma_X^*$ records their graded interchange.
\end{theorem}
\begin{proof}
It suffices to use $\theta=\alpha\otimes\beta$ with homogeneous cohomology classes. Proposition~\ref{prop:Chern-shifts} gives the following identity of maps from $H^*(X^{[n]})$ to $H^*(X^{[n+1]})$:
\begin{equation}\label{eq:Gamma-E-colored}
\begin{split}
&\mathfrak G_{k,n+1}(\alpha)E_{j,n}(\beta)
-(-1)^{\bar\alpha\bar\beta}E_{j,n}(\beta)\mathfrak G_{k,n}(\alpha)\\
&\hspace{35mm}=\frac1{k!}E_{j+k,n}(\alpha\beta).
\end{split}
\end{equation}
By \eqref{eq:left-insertion}, the class $\log H_n(z)$ likewise gives a family $\bigl(\log H_n(z)\bigr)(\alpha)$ of cup-product operators. The term $\log(1+K_X/z)$ in \eqref{eq:logH}
is independent of $n$ and acts by multiplication
by the same scalar on every $H^*(X^{[n]})$.
Hence its contributions cancel in the commutator. Substitution of \eqref{eq:Gamma-E-colored} into \eqref{eq:logH} therefore gives
\begin{align}
&\bigl(\log H_{n+1}(z)\bigr)(\alpha)E_{j,n}(\beta)
-(-1)^{\bar\alpha\bar\beta}E_{j,n}(\beta)\bigl(\log H_n(z)\bigr)(\alpha)\notag\\
&\qquad=-\sum_{k\geq0}\frac{(k+1)(k+2)}{z^{k+3}}
E_{j+k,n}(K_X\alpha\beta).
\label{eq:log-exchange}
\end{align}
The generating identities are
\begin{equation}\label{eq:shift-kernel}
\begin{split}
\frac2{(z-w)^3}&=\sum_{k\geq0}(k+1)(k+2)\frac{w^k}{z^{k+3}},\\
[w^kE_n(w;\gamma)]_{w<0}&=\sum_{j\geq0}E_{j+k,n}(\gamma)w^{-j-1}.
\end{split}
\end{equation}
Thus the right side of the current form of \eqref{eq:log-exchange} is
$[-2E_n(w;K_X\alpha\beta)/(z-w)^3]_{z\gg w,\,w<0}$.

We exponentiate at the level of classes. Choose a homogeneous K\"unneth decomposition
$\tau_{2*}\alpha=\sum_\ell\alpha_\ell'\otimes\alpha_\ell''$.
By repeated use of \eqref{eq:log-exchange} and Lemma~\ref{lem:products}, for $q\geq1$,
\begin{equation}\label{eq:iterated-log-exchange}
\begin{aligned}
&\sum_{j=0}^q(-1)^j\binom qj\sum_\ell
(-1)^{\bar\alpha_\ell''\bar\beta}
\bigl((\log H_{n+1}(z))^{q-j}\bigr)(\alpha_\ell')\\
&\hspace{22mm}\cdot E_n(w;\beta)
\bigl((\log H_n(z))^j\bigr)(\alpha_\ell'')\\
&\qquad=\left[\frac{(-2)^q}{(z-w)^{3q}}
E_n(w;K_X^qe_X^{q-1}\alpha\beta)\right]_{z\gg w,\,w<0}.
\end{aligned}
\end{equation}
Indeed, the left side is the binomial expansion of the $q$ nested logarithmic commutators, with their cohomology classes distributed by $\tau_{q*}\alpha$. Each commutator contributes one $K_X$ and the shift in \eqref{eq:shift-kernel}. Indeed, the self-intersection formula for the small diagonal gives
\[
\tau_q^*\tau_{q*}\alpha=e_X^{q-1}\alpha.
\]
Since each application of \eqref{eq:log-exchange} introduces a factor of $K_X$, the $q$-fold expression involves $K_X^q e_X^{q-1}\alpha\beta$.
This vanishes for $q\ge2$, since $K_Xe_X=0$.

The factor
\[
(-1)^{\bar\alpha_\ell''\bar\beta}
\]
in \eqref{eq:iterated-log-exchange} accounts for moving the cohomology class $\alpha_\ell''$ past $\beta$. The separate factor $(-1)^j$ is the binomial sign.

It is also legitimate to project to negative powers of $w$ after each
application of \eqref{eq:log-exchange}. If
\[
P_-f=[f]_{w<0},
\]
then, for every $k\geq0$,
\begin{equation}\label{eq:w-projection}
P_-\bigl(w^kP_-f\bigr)=P_-\bigl(w^kf\bigr).
\end{equation}
Indeed, a term whose $w$-degree is already nonnegative remains
nonnegative after multiplication by $w^k$. By \eqref{eq:shift-kernel}, every subsequent shift in the iteration has this form. Thus a discarded term cannot contribute to the final negative-power part.

We sum \eqref{eq:iterated-log-exchange} with coefficients $1/q!$ and include the $q=0$ term. The left-hand side exponentiates at the level of classes. Setting
$a=q-j$ and $b=j$, and using
\[
\frac1{q!}\binom qj=\frac1{a!\,b!},
\]
gives
\begin{equation}\label{eq:exponentiate-log-exchange}
\begin{aligned}
&\sum_{q\geq0}\frac1{q!}\sum_{j=0}^q(-1)^j\binom qj
\sum_\ell(-1)^{\bar\alpha_\ell''\bar\beta}
\bigl((\log H_{n+1}(z))^{q-j}\bigr)(\alpha_\ell')\\
&\hspace{20mm}\cdot E_n(w;\beta)
\bigl((\log H_n(z))^j\bigr)(\alpha_\ell'')\\
&\quad=\sum_\ell(-1)^{\bar\alpha_\ell''\bar\beta}
\sum_{a,b\geq0}\frac{(-1)^b}{a!b!}
\bigl((\log H_{n+1}(z))^a\bigr)(\alpha_\ell')\\
&\hspace{20mm}\cdot E_n(w;\beta)
\bigl((\log H_n(z))^b\bigr)(\alpha_\ell'')\\
&\quad=\sum_\ell(-1)^{\bar\alpha_\ell''\bar\beta}
H_{n+1}(z;\alpha_\ell')E_n(w;\beta)
\bigl(H_n(z)^{-1}\bigr)(\alpha_\ell'').
\end{aligned}
\end{equation}

By Lemma~\ref{lem:zeta}, the right side of \eqref{eq:iterated-log-exchange} vanishes for $q\geq2$. Thus only $q=0,1$ remain in the preceding sum. The quotient identity in the same lemma gives
\begin{equation}\label{eq:employ-zeta-lemma}
\begin{aligned}
&\sum_\ell(-1)^{\bar\alpha_\ell''\bar\beta}
H_{n+1}(z;\alpha_\ell')E_n(w;\beta)
\bigl(H_n(z)^{-1}\bigr)(\alpha_\ell'')\\
&\quad=\left[\left(\int_X\alpha\right)E_n(w;\beta)
-\frac{2E_n(w;K_X\alpha\beta)}{(z-w)^3}\right]_{z\gg w,\,w<0}\\
&\quad=\left[(\mathbf 1\,E_n(w))\left(
(\alpha\otimes\beta)\cup\frac{\zeta(z-w)}{\zeta(w-z)}
\right)\right]_{z\gg w,\,w<0}.
\end{aligned}
\end{equation}
The projection formula and the counit identity
$\sum_\ell(\int_X\eta_\ell')\eta_\ell''=\eta$ for $\tau_{2*}\eta=\sum_\ell\eta_\ell'\otimes\eta_\ell''$ explain this last equality. 

By \eqref{eq:employ-zeta-lemma} and
$H_n(z)^{-1}H_n(z)=1$, we have, for
$\theta=\alpha\otimes\beta$ and $s=z-w$,
\begin{align*}
&(H_{n+1}(z)E_n(w))(\theta)\\
&\quad=
\bigl((H_{n+1}(z)E_n(w)H_n(z)^{-1})
H_n(z)\bigr)(\theta)\\
&\quad=
\left[
(-1)^{\bar\alpha\bar\beta}
E_n(w;\beta)H_n(z;\alpha)
-\frac{2}{s^3}
(E_n(w)H_n(z))
\bigl(\Delta_{X*}(K_X\alpha\beta)\bigr)
\right]_{z\gg w,\,z^{\leq0},\,w<0}\\
&\quad=
\left[
(E_n(w)H_n(z))
\left(
\sigma_X^*\left(
\theta\cup\frac{\zeta(s)}{\zeta(-s)}
\right)
\right)
\right]_{z\gg w,\,z^{\leq0},\,w<0}.
\end{align*}
Here the products are understood using the small diagonal,
as in Lemma~\ref{lem:products}. The last equality follows from
\[
\sigma_X^*(\alpha\otimes\beta)
=(-1)^{\bar\alpha\bar\beta}\beta\otimes\alpha,
\qquad
\sigma_X^*\bigl(
(\alpha\otimes\beta)\cup\Delta_{X*}K_X
\bigr)
=\Delta_{X*}(K_X\alpha\beta).
\]
This proves \eqref{eq:main-HE}.

The relation \eqref{eq:main-FH} follows by the same argument,
using the commutator formula for $F_j$ in
\eqref{eq:Chern-shifts}.
The interchange of the two cohomology classes
is accounted for by $\sigma_X^*$.
\end{proof}

\subsubsection{Exchange relations for the regular Cartan current}
The identities just proved concern $H_n=1+K_X\mathscr C_n$. Multiplication by $K_X$ loses information, so these identities alone do not imply the corresponding relations for $\mathscr C_n$. We first derive the relations using a formal parameter \(u\), and then specialize to \(u=K_X\).

\begin{proposition}\label{prop:divided-exchange}
Set $s=z-w$. 
For homogeneous $\alpha,\beta$,
\begin{align}
&\mathscr C_{n+1}(z;\alpha)E_n(w;\beta)
-(-1)^{\bar\alpha\bar\beta}E_n(w;\beta)\mathscr C_n(z;\alpha)\notag\\
&\quad=\left[(E_n(w)H_n(z))\left(
\tau_{2*}\left(\left(-\frac2{(z-w)^3}
+\frac{2(e_X-K_X^2)}{(z-w)^5}\right)\alpha\beta\right)
\right)\right]_{z\gg w,\,w<0}.
\label{eq:divided-HE}
\end{align}
The lowering identity has left side
$F_n(w;\beta)\mathscr C_n(z;\alpha)-(-1)^{\bar\alpha\bar\beta}\mathscr C_{n-1}(z;\alpha)F_n(w;\beta)$;
on the right replace the product by $H_{n-1}(z)F_n(w)$ and the cohomology class $\alpha\beta$ by $\beta\alpha$. All products are evaluated on their displayed input classes, not on retained output factors.
\end{proposition}
\begin{proof}
Retain the independent shift and set
$H_n(z;u)=\exp(P_n(z+u)-P_n(z))$ as a class-valued series. For $\gamma\in H^*(X)$, denote the corresponding cup-product operator by $H_n(z;u;\gamma)$. Equations \eqref{eq:P} and \eqref{eq:Gamma-E-colored} show that the commutator of $P_n(z;\alpha)$ with $E_n(w;\beta)$ depends on $\alpha$ and $\beta$ through the kernel
\begin{equation}\label{eq:P-E-kernel}
[\Delta_X]\left(\frac1{s^2}+\frac{K_1}{s^3}
+\frac{K_1^2-e_1/2}{s^4}\right)
=\log\left(1+\frac{[\Delta_X]}{s(s-K_1)}\right).
\end{equation}
For example, its evaluation on $\alpha\otimes\beta$ is
\[
\left[E_n\left(w;\left(\frac1{s^2}+\frac{K_X}{s^3}
+\frac{K_X^2-e_X/2}{s^4}\right)\alpha\beta\right)\right]_{z\gg w,\,w<0}.
\]
The factorials in \eqref{eq:P}, divided by $k!$, give these three rational terms. The logarithm in \eqref{eq:P-E-kernel} follows from
$[\Delta_X]^2=[\Delta_X]e_1$ and $[\Delta_X]^3=0$.
The same expansion, with the Cartan cohomology classes combined as in \eqref{eq:iterated-log-exchange}, therefore proves, for every $\theta\in H^*(X\times X)$,
\begin{equation}\label{eq:formal-shift-HE}
\begin{split}
&(H_{n+1}(z;u)E_n(w))(\theta)\\
&\quad=\left[(E_n(w)H_n(z;u))\left(
\sigma_X^*\left(\theta\cup
\frac{1+[\Delta_X]/((s+u)(s+u-K_1))}
{1+[\Delta_X]/(s(s-K_1))}\right)\right)\right]_{z\gg w,\,w<0}.
\end{split}
\end{equation}
All coefficients are formal series in $u$ before specialization. This step does not use the undivided relation after $u=K_X$.

We now derive the relation for $\mathscr C_n(z)$
from \eqref{eq:formal-shift-HE}.
Since
$H_n(z;u)=1+u\mathscr C_n(z;u)$,
we subtract the term corresponding to $1$ in the kernel
and divide by $u$, before setting $u=K_X$.

To compute the remaining kernel, put
$b(s)=s^{-2}+K_1s^{-3}+(K_1^2-e_1/2)s^{-4}$ within this calculation. Recall from
\eqref{eq:P-E-kernel} that
\[
\log\left(1+\frac{[\Delta_X]}{s(s-K_1)}\right)
=[\Delta_X]b(s),
\]
where
\[
b(s)=\frac1{s^2}+\frac{K_1}{s^3}
+\frac{K_1^2-e_1/2}{s^4}.
\]
Hence the quotient of kernels in
\eqref{eq:formal-shift-HE} is
\[
\begin{aligned}
&\exp\left(
[\Delta_X](b(s+u)-b(s))
\right)\\
&\quad=1+[\Delta_X](b(s+u)-b(s))
+\frac12[\Delta_X]e_1(b(s+u)-b(s))^2,
\end{aligned}
\]
since $[\Delta_X]^2=[\Delta_X]e_1$
and $[\Delta_X]^3=0$.

Since $b(s+u)-b(s)$ is divisible by $u$, we may
divide by $u$ before setting $u=K_1$.
Taylor expansion gives
\begin{align*}
\left.\frac{b(s+u)-b(s)}{u}\right|_{u=K_1}
&=b'(s)+\frac{K_1}{2}b''(s)
+\frac{K_1^2}{6}b^{(3)}(s)\\
&=-\frac{2}{s^3}
+\frac{2(e_1-K_1^2)}{s^5}.
\end{align*}
The second term vanishes after specialization,
since
\[
\frac{e_1(b(s+u)-b(s))^2}{u}
\]
is divisible by $u$ and $K_1e_1=0$.

Thus, after subtracting $1$, dividing by $u$,
and setting $u=K_1$, the kernel in
\eqref{eq:formal-shift-HE} becomes
\[
[\Delta_X]\left(
-\frac{2}{s^3}
+\frac{2(e_1-K_1^2)}{s^5}
\right).
\]
For $\theta=\alpha\otimes\beta$, the projection formula gives
\[
(\alpha\otimes\beta)\cup\Delta_{X*}\eta
=\tau_{2*}(\eta\alpha\beta).
\]
Applying this identity to \eqref{eq:formal-shift-HE}
yields \eqref{eq:divided-HE}.

The relation for $F_n$ follows by taking the graded
adjoint, with the order of the cohomology classes
reversed.
\end{proof}

\subsection{Quadratic relations for \(E\) and \(F\)}\label{sec:same-sign}
We prove \eqref{eq:main-EE} and \eqref{eq:main-FF} using the normal-ordered boundary derivatives, the Heisenberg relations, and Proposition~\ref{prop:Chern-shifts}. We reduce the positive relation to polynomial mode identities, check the lowest one explicitly, and propagate it to all indices. The negative relation then follows by adjoint.

\subsubsection{Transfer property }
The transfer property was introduced in \cite[Definition~3.16]{LQW1} and established there for commutators involving higher derivatives of Nakajima operators. We verify the corresponding property for the operators \(E_k\) below.

For homogeneous $\gamma$, set
\begin{equation}\label{eq:R-def}
R_{a,b}(\gamma)=[E_a(1_X),E_b(\gamma)].
\end{equation}

\begin{lemma}\label{lem:EE-transfer}
For all homogeneous $\alpha,\beta$,
\begin{equation}\label{eq:EE-transfer}
[E_a(\alpha),E_b(\beta)]=R_{a,b}(\alpha\beta),\qquad
R_{b,a}=-R_{a,b}.
\end{equation}
Equivalently, for $\theta\in H^*(X\times X)$,
\[
(E_aE_b)(\theta)-(E_bE_a)(\sigma_X^*\theta)
=R_{a,b}(\Delta_X^*\theta).
\]
In particular, the difference vanishes when $\theta$ is replaced by $(K_1-K_2)\theta$.
\end{lemma}
\begin{proof}
By \cite[Proposition~3.17]{LQW1}, the transfer property
holds for $[E_a(\alpha),E_0(\beta)]$ for every $a\geq0$.
Since $E_{b+1}=[\mathfrak d,E_b]$, the Jacobi identity gives
\[
\begin{aligned}
[E_a(\alpha),E_{b+1}(\beta)]
={}&[\mathfrak d,[E_a(\alpha),E_b(\beta)]]\\
&-[E_{a+1}(\alpha),E_b(\beta)].
\end{aligned}
\]
Induction on $b$ therefore proves the transfer property
for all $a,b\geq0$.

The identity $R_{b,a}=-R_{a,b}$ follows from
super-skew-symmetry. The remaining assertions follow
by linearity from
$\Delta_X^*(\alpha\otimes\beta)=\alpha\beta$
and $\Delta_X^*(K_1-K_2)=0$.
\end{proof}

For $\tau_{2*}\gamma=\sum_i\gamma_i'\otimes\gamma_i''$, write
\begin{equation}\label{eq:diag-composition}
(E_aE_b)_\Delta(\gamma)=\sum_iE_a(\gamma_i')E_b(\gamma_i'').
\end{equation}
In the product convention of Subsection~\ref{subsec:current-conventions}, this is simply $(E_aE_b)(\tau_{2*}\gamma)$. The graded cocommutativity of $\tau_{2*}$ and Lemma~\ref{lem:EE-transfer} imply
\begin{equation}\label{eq:diag-difference}
(E_aE_b)_\Delta(\gamma)-(E_bE_a)_\Delta(\gamma)=R_{a,b}(e_X\gamma).
\end{equation}
Indeed, after the graded permutation the difference is the sum of supercommutators, whose cohomology class is
$\sum_i\gamma_i'\gamma_i''=\Delta_X^*\Delta_{X*}\gamma=e_X\gamma$.

\subsubsection{The polynomial mode relation}
Recall the kernels \eqref{eq:zeta}--\eqref{eq:tilde-zeta}. For $\theta\in H^*(X\times X)$, set
\begin{equation}\label{eq:EE-Q-current}
\begin{aligned}
\mathcal Q_n(z,w;\theta)
:={}&
(E_{n+1}(z)E_n(w))
\bigl(\theta\cup\widetilde\zeta_-(w-z)\bigr)
\\
&-
(E_{n+1}(w)E_n(z))
\bigl(
\sigma_X^*
(\theta\cup\widetilde\zeta_+(z-w))
\bigr).
\end{aligned}
\end{equation}
The positive relation \eqref{eq:main-EE} means
\[
[(z-w)\mathcal Q_n(z,w;\theta)]_{z<0,w<0}=0.
\]
Thus we clear the denominator before projecting. Put $s=z-w$.
Direct expansion gives
\begin{align}
s\widetilde\zeta_+(s)
&=s^3+(K_1-K_2)s^2-K_1K_2s+[\Delta_X](s-K_1),
\label{eq:kernel-clear-plus}\\
s\widetilde\zeta_-(-s)
&=s^3+(K_1-K_2)s^2-K_1K_2s+[\Delta_X](s+K_1).
\label{eq:kernel-clear-minus}
\end{align}
This is a polynomial identity in $z,w$, unlike the expansion for
$z\gg w$ used in the Cartan exchange relations.

Take $\theta=\alpha\otimes\beta$ with homogeneous $\alpha,\beta$, and
write
\[
\gamma=\alpha\beta.
\]
By Lemma~\ref{lem:EE-transfer}, the difference of the two ordered
products depends on $\alpha$ and $\beta$ only through their cup product:
\[
(E_aE_b)(\alpha\otimes\beta)
-
(E_bE_a)\bigl(\sigma_X^*(\alpha\otimes\beta)\bigr)
=
R_{a,b}(\gamma).
\]
Consequently, the term proportional to $K_1-K_2$ in
\eqref{eq:kernel-clear-plus}--\eqref{eq:kernel-clear-minus} vanishes,
because
\[
\Delta_X^*(K_1-K_2)=0.
\]
The factor $K_1K_2$ restricts to $K_X^2$ on the diagonal and hence
contributes
\[
-s\sum_{a,b\geq0}
R_{a,b}(K_X^2\gamma)z^{-a-1}w^{-b-1}.
\]
On the other hand, the terms containing $[\Delta_X]s$ give, by
\eqref{eq:diag-difference},
\[
s\sum_{a,b\geq0}
R_{a,b}(e_X\gamma)z^{-a-1}w^{-b-1}.
\]
Finally, the two terms containing $[\Delta_X]K_1$ have the same sign
after taking the difference in \eqref{eq:EE-Q-current}, and give
\[
\sum_{a,b\geq0}
\left(
(E_aE_b)_\Delta(K_X\gamma)
+
(E_bE_a)_\Delta(K_X\gamma)
\right)
z^{-a-1}w^{-b-1}.
\]

Restricting every operator on the right to source $n$, we obtain
\begin{equation}\label{eq:EE-Q-expanded}
\begin{aligned}
s\mathcal Q_n(z,w;\alpha\otimes\beta)
={}&
s^3\sum_{a,b\geq0}
R_{a,b}(\gamma)z^{-a-1}w^{-b-1}
\\
&+
s\sum_{a,b\geq0}
R_{a,b}\bigl((e_X-K_X^2)\gamma\bigr)
z^{-a-1}w^{-b-1}
\\
&+
\sum_{a,b\geq0}
\left(
(E_aE_b)_\Delta(K_X\gamma)
+
(E_bE_a)_\Delta(K_X\gamma)
\right)
z^{-a-1}w^{-b-1}.
\end{aligned}
\end{equation}

We now define $Q_{a,b}(\gamma)$ to be the coefficient of
$z^{-a-1}w^{-b-1}$ in the right-hand side of
\eqref{eq:EE-Q-expanded}. Since
\[
s^3=(z-w)^3=z^3-3z^2w+3zw^2-w^3,
\]
and
\[
s=z-w,
\]
we obtain
\begin{align}
Q_{a,b}(\gamma):={}&
R_{a+3,b}(\gamma)
-3R_{a+2,b+1}(\gamma)
+3R_{a+1,b+2}(\gamma)
-R_{a,b+3}(\gamma)
\notag\\
&+
R_{a+1,b}\bigl((e_X-K_X^2)\gamma\bigr)
-
R_{a,b+1}\bigl((e_X-K_X^2)\gamma\bigr)
\notag\\
&+
(E_aE_b)_\Delta(K_X\gamma)
+
(E_bE_a)_\Delta(K_X\gamma).
\label{eq:EE-mode}
\end{align}
Consequently, \eqref{eq:main-EE} is equivalent to
\begin{equation}\label{eq:EE-mode-vanishing}
Q_{a,b}(\gamma)=0,\qquad a,b\geq0,
\end{equation}
for every homogeneous $\gamma\in H^*(X)$.


\subsubsection{The lowest relation}
\begin{lemma}\label{lem:EE-base}
For every homogeneous $\gamma$, $Q_{0,0}(\gamma)=0$.
\end{lemma}
\begin{proof}
Equations \eqref{eq:derivative-expansion} and \eqref{eq:small-f} yield
\begin{equation}\label{eq:E2-EE}
E_2(\gamma)=2\sum_{\substack{\ell(\nu)=3\\|\nu|=-1}}
\frac{\mathfrak a_\nu(\tau_*\gamma)}{\nu!}
+\sum_{j\geq1}j\,\mathfrak a_{-j-1}\mathfrak a_j(\tau_{2*}(K_X\gamma)).
\end{equation}
There is no Euler or $K_X^2$ term here: their possible singleton mode is $-1$ and both coefficients are zero. Since $E_1(1_X)=-\mathfrak L_{-1}(1_X)$ shifts a mode $q$ to $q-1$ with coefficient $q$, it introduces no new contraction. On a target triple of total mode $-2$, the sum of the three coefficients $q_i+1$ is one. The first term of \eqref{eq:E2-EE} therefore keeps coefficient two. On the second term, the contributions to $\mathfrak a_{-j-2}\mathfrak a_j$ have coefficient
$-j(j+1)+(j+1)^2=j+1$.
Thus
\begin{equation}\label{eq:R12-explicit}
R_{1,2}(\gamma)=2\sum_{\substack{\ell(\nu)=3\\|\nu|=-2}}
\frac{\mathfrak a_\nu(\tau_*\gamma)}{\nu!}
+\sum_{j\geq1}(j+1)\mathfrak a_{-j-2}\mathfrak a_j(\tau_{2*}(K_X\gamma)).
\end{equation}
At total weight $n=2$, the recursions give
\[
\fk{2}{(-1,-1)}=2,\qquad
\fk{2}{(-j-2,j)}=4(j+1),\qquad
\fkk{2}{(-2)}=1.
\]
The leading second-derivative coefficient is $2^2\cdot2!=8$, and the Euler singleton coefficient is
$-8(4-1)/24=-1$. Hence the full expansion is
\begin{align}
\mathfrak a_{-2}^{(2)}(\gamma)={}&8\sum_{\substack{\ell(\nu)=3\\|\nu|=-2}}
\frac{\mathfrak a_\nu(\tau_*\gamma)}{\nu!}
+4\sum_{j\geq1}(j+1)\mathfrak a_{-j-2}\mathfrak a_j(\tau_{2*}(K_X\gamma))\notag\\
&+\mathfrak a_{-1}^2(\tau_{2*}(K_X\gamma))
+\mathfrak a_{-2}((K_X^2-e_X)\gamma).
\label{eq:a-minus2-second}
\end{align}
The coefficient of the repeated-mode term is one because its divided-monomial denominator is $2!$.

Now $R_{0,1}(\gamma)=\mathfrak a_{-2}(\gamma)$ and
$R_{1,1}(\gamma)=0$ by the Heisenberg--Virasoro relations. Differentiation gives
\begin{equation}\label{eq:R03-reduction}
R_{0,2}(\gamma)=\mathfrak a_{-2}'(\gamma),\qquad
R_{0,3}(\gamma)=\mathfrak a_{-2}^{(2)}(\gamma)-R_{1,2}(\gamma).
\end{equation}
Using antisymmetry and $(E_0E_0)_\Delta(K_X\gamma)=\mathfrak a_{-1}^2(\tau_{2*}(K_X\gamma))$, we obtain
\begin{align}
\tfrac12Q_{0,0}(\gamma)
={}&4R_{1,2}(\gamma)-\mathfrak a_{-2}^{(2)}(\gamma)
-\mathfrak a_{-2}((e_X-K_X^2)\gamma)
+\mathfrak a_{-1}^2(\tau_{2*}(K_X\gamma))\notag\\
={}&0
\label{eq:EE-base-cancellation}
\end{align}
by \eqref{eq:R12-explicit}--\eqref{eq:a-minus2-second}. This displays the cancellation of the length-three, length-two, and singleton contributions separately. It is valid for odd external $\gamma$ as well, since all operators used for the shifts involve the even class $1_X$.
\end{proof}

\subsubsection{Propagation to all levels}

\begin{theorem}\label{thm:EE-independent}
For all $a,b\geq0$ and homogeneous $\gamma$, one has $Q_{a,b}(\gamma)=0$.
Consequently, the  relations \eqref{eq:main-EE} and \eqref{eq:main-FF}
hold in every source sector in which they are defined.
\end{theorem}

\begin{proof}
Put $\partial_r=\ad(r!\mathfrak G_r(1_X))$. Equation \eqref{eq:Chern-shifts} and the derivation rule imply
\[
\partial_rR_{a,b}=R_{a+r,b}+R_{a,b+r},\qquad
\partial_r(E_aE_b)_\Delta=(E_{a+r}E_b)_\Delta+(E_aE_{b+r})_\Delta.
\]
Consequently
\begin{equation}\label{eq:Q-propagation}
\partial_rQ_{a,b}=Q_{a+r,b}+Q_{a,b+r},\qquad Q_{b,a}=Q_{a,b}.
\end{equation}
The second equality follows by antisymmetry of $R$ and symmetry of the last line of \eqref{eq:EE-mode}. Apply $\partial_r$ to the identity of Lemma~\ref{lem:EE-base}. It gives $2Q_{r,0}=0$, hence $Q_{r,0}=0$. Applying $\partial_b$ to $Q_{a,0}=0$ then gives
$0=Q_{a+b,0}+Q_{a,b}=Q_{a,b}$.
This proves the statement for all mode indices and all cohomology classes $\gamma$, without an induction hypothesis on unknown coefficients.

By construction, \(Q_{a,b}(\gamma)\) is the coefficient of \(z^{-a-1}w^{-b-1}\) in the denominator-cleared \(EE\) current relation. Equations \eqref{eq:kernel-clear-plus}--\eqref{eq:kernel-clear-minus}, together with the transfer identity and \eqref{eq:diag-difference}, identify the two formulations on both sides of \eqref{eq:main-EE}. Hence \(Q_{a,b}=0\) for all \(a,b\) proves the \(EE\) relation. The \(FF\) relation follows by taking the graded adjoint, which interchanges the two cohomology classes via \(\sigma_X^*\).
\end{proof}

\subsection{The Serre relations}\label{sec:Serre}
It remains to prove the Serre relations. As in the quadratic argument,
the Chern shifts extend a lowest-mode identity to all indices.

For homogeneous cohomology classes $\alpha_1,\alpha_2,\alpha_3$, let
\begin{equation}\label{eq:Koszul}
\epsilon(\sigma;\alpha_1,\alpha_2,\alpha_3)
=\prod_{\substack{i<j\\\sigma(i)>\sigma(j)}}
(-1)^{\bar\alpha_{\sigma(i)}\bar\alpha_{\sigma(j)}},\qquad \sigma\in S_3.
\end{equation}
The mode indices and their cohomology classes are permuted together.
\begin{theorem}\label{thm:Serre}
For $d_1,d_2,d_3,n\geq0$ and homogeneous $\alpha_1,\alpha_2,\alpha_3\in H^*(X)$,
\begin{equation}\label{eq:Serre}
\left.\sum_{\sigma\in S_3}\epsilon(\sigma;\alpha_1,\alpha_2,\alpha_3)
\left[E_{d_{\sigma(1)}}(\alpha_{\sigma(1)}),
\left[E_{d_{\sigma(2)}}(\alpha_{\sigma(2)}),
E_{d_{\sigma(3)}+1}(\alpha_{\sigma(3)})\right]\right]
\right|_{H^*(X^{[n]})}=0.
\end{equation}
The analogous identity holds with every $E$ replaced by $F$.
\end{theorem}
\begin{proof}
First suppose the cohomology classes are even. The commuting derivations
$\mathcal D_{r,\eta}=r!\ad\mathfrak G_r(\eta)$ satisfy
\[
\mathcal D_{r,\eta}E_j(\beta)=E_{j+r}(\eta\beta),\qquad
\mathcal D_{a,\alpha}\mathcal D_{b,\beta}E_j(\eta)
=\mathcal D_{a+b,\alpha\beta}E_j(\eta).
\]
The second statement is an identity on a single generator, not an identity of operations on arbitrary products. By \eqref{eq:La},
$[E_0(1_X),E_1(1_X)]=\mathfrak a_{-2}(1_X)$, whence
\begin{equation}\label{eq:lowest-Serre}
[E_0(1_X),[E_0(1_X),E_1(1_X)]]=0.
\end{equation}
Apply to this zero identity the operation
\begin{align}
&\mathcal D_{a,\alpha}\mathcal D_{b,\beta}\mathcal D_{c,\gamma}
-\mathcal D_{a+b,\alpha\beta}\mathcal D_{c,\gamma}
-\mathcal D_{a+c,\alpha\gamma}\mathcal D_{b,\beta}\notag\\
&\hspace{12mm}-\mathcal D_{b+c,\beta\gamma}\mathcal D_{a,\alpha}
+2\mathcal D_{a+b+c,\alpha\beta\gamma}.
\label{eq:Serre-polarization}
\end{align}
Leibniz expansion distributes three labelled derivations among three generator positions. Terms with exactly two on one position cancel against their merged derivation. Terms with all three on one position have coefficient $1-3+2=0$. The six assignments to distinct positions remain. The third position initially has index one, so they are exactly the six terms of \eqref{eq:Serre} with the innermost index raised by one. Repeated or zero shifts cause no problem because the derivations are labelled.

To allow odd cohomology classes, adjoin independent supercommuting variables $\theta_i$ of parity $\bar\alpha_i$, apply the even statement to $\theta_i\alpha_i$, and extract the coefficient of $\theta_1\theta_2\theta_3$ in that order. The Koszul signs are exactly \eqref{eq:Koszul}, up to one common sign that cancels. Finally,
$[A,[B,C]]^\dagger=[A^\dagger,[B^\dagger,C^\dagger]]$
by \eqref{eq:graded-adjoint}; taking adjoints proves the negative statement.
\end{proof}

\subsection{The affine Yangian action}\label{sec:action}
We now show that the operator relations established above define an affine Yangian action. There are two points to address. First, the parameters of the Yangian will be identified with the characteristic classes \(K_X\) and \(e_X\), which are not invertible. We therefore use a presentation that does not require division by these classes. Second, our operators depend on cohomology classes in \(H^*(X)\), so we must specify how their multiplication and commutators correspond to those of the Yangian. We discuss these two points in the following subsections.

\subsubsection{The presentation and its parameters}
We first specify the algebraic relations to be realized by the
geometric operators. The next subsection will define the corresponding
multiplication and bracket on $H^*(X)$-parametrized operators.

Put
\[
\Lambda=\mathbb Q[t_1,t_2]^{\mathfrak S_2}
=\mathbb Q[t,\varepsilon],
\qquad t=t_1+t_2,\qquad \varepsilon=t_1t_2,
\]
and set
\[
V(s)=\frac{-2}{s^3+(\varepsilon-t^2)s+\varepsilon t}.
\]
The series $V$ is expanded at infinity, so every
coefficient lies in $\Lambda$. Under this specialization, we have
\begin{equation}\label{eq:V-surface}
\left.V(s)\right|_{t=K_X,\,\varepsilon=e_X}
=-\frac{2}{s^3}
+\frac{2(e_X-K_X^2)}{s^5},
\end{equation}
since $K_Xe_X=0$ and $(e_X-K_X^2)^2=0$.
This agrees with the kernel obtained in
Proposition~\ref{prop:divided-exchange}. Consider generators $e_k,f_k,c_k$ for $k\geq0$, and write
\[
\begin{aligned}
e(z)&=\sum_{k\geq0}e_kz^{-k-1},
&f(z)&=\sum_{k\geq0}f_kz^{-k-1},\\
c(z)&=\sum_{k\geq0}c_kz^{-k-1},
&h(z)&=1+t c(z).
\end{aligned}
\]
The $c_k$ are generators in their own right, since $h(z)$ does not
determine them when $t=0$.

Since $\varepsilon$ need not be invertible, we seek a
$\Lambda$-bilinear skew-symmetric bracket
$[\cdot,\cdot]_\varepsilon$ on a unital associative
$\Lambda$-algebra generated by $e_k,f_k,c_k$.
For arbitrary elements $\xi,\eta,\rho$ of this algebra,
we require
\begin{equation}\label{eq:divided-algebra-rules}
\begin{gathered}
\varepsilon[\xi,\eta]_\varepsilon
=\xi\eta-\eta\xi,\qquad
[\xi,\eta\rho]_\varepsilon
=[\xi,\eta]_\varepsilon\rho
+\eta[\xi,\rho]_\varepsilon,\\
[\xi,[\eta,\rho]_\varepsilon]_\varepsilon
+[\eta,[\rho,\xi]_\varepsilon]_\varepsilon
+[\rho,[\xi,\eta]_\varepsilon]_\varepsilon
=0.
\end{gathered}
\end{equation}
These identities are required for all $\xi,\eta,\rho$. In particular, the
first identity is a relation between two operations, not a definition
of the bracket by division by $\varepsilon$.

We impose the following additional relations on the generators:
\begin{align}
[c_k,c_\ell]_\varepsilon&=0,\qquad
[e_k,f_\ell]_\varepsilon=-c_{k+\ell},
\label{eq:abstract-mixed}\\
[c(z),e(w)]_\varepsilon
&=[e(w)h(z)V(z-w)]_{z\gg w,\,w<0},
\label{eq:abstract-CE}\\
[f(w),c(z)]_\varepsilon
&=[h(z)f(w)V(z-w)]_{z\gg w,\,w<0}.
\label{eq:abstract-FC}
\end{align}
The right-hand sides of the last two identities contain only negative
powers of $z$. The quadratic relation for the $e_k$ is
\begin{align}
0={}&[e_{k+3},e_\ell]_\varepsilon
-3[e_{k+2},e_{\ell+1}]_\varepsilon
+3[e_{k+1},e_{\ell+2}]_\varepsilon
-[e_k,e_{\ell+3}]_\varepsilon\notag\\
&+(\varepsilon-t^2)
\bigl([e_{k+1},e_\ell]_\varepsilon
-[e_k,e_{\ell+1}]_\varepsilon\bigr)
+t(e_ke_\ell+e_\ell e_k).
\label{eq:abstract-EE}
\end{align}
The relation for the $f_k$ is obtained by replacing every $e$ by $f$
and reversing the sign of the last term. We also impose the Serre
relations
\begin{equation}\label{eq:abstract-Serre}
\begin{aligned}
\sum_{\sigma\in\mathfrak S_3}
[e_{d_{\sigma(1)}},
[e_{d_{\sigma(2)}},e_{d_{\sigma(3)}+1}]_\varepsilon]_\varepsilon
&=0,\\
\sum_{\sigma\in\mathfrak S_3}
[f_{d_{\sigma(1)}},
[f_{d_{\sigma(2)}},f_{d_{\sigma(3)}+1}]_\varepsilon]_\varepsilon
&=0.
\end{aligned}
\end{equation}
All indices in these relations are nonnegative.
Equations~\eqref{eq:abstract-CE}--\eqref{eq:abstract-FC} include the
relations for the first Cartan coefficients. In particular,
\[
\begin{aligned}
[c_0,e_k]_\varepsilon=[c_1,e_k]_\varepsilon&=0,
&[c_2,e_k]_\varepsilon&=-2e_k,\\
[c_0,f_k]_\varepsilon=[c_1,f_k]_\varepsilon&=0,
&[c_2,f_k]_\varepsilon&=2f_k.
\end{aligned}
\]

To compare with the usual affine Yangian presentation, invert
$\varepsilon$. Then $[\xi,\eta]_\varepsilon=\varepsilon^{-1}(\xi\eta-\eta\xi)$.
Set
\begin{equation}\label{eq:standard-normalization}
\mathrm e_k=\varepsilon^{-1}e_k,\qquad
\mathrm f_k=\varepsilon^{-1}f_k,\qquad
\psi_k=-\varepsilon^{-1}c_k.
\end{equation}
With $h_1=t_1$, $h_2=t_2$, and $h_3=-t$, we have
\[
\sigma_2=h_1h_2+h_1h_3+h_2h_3=\varepsilon-t^2,
\qquad
\sigma_3=h_1h_2h_3=-\varepsilon t.
\]
Under \eqref{eq:standard-normalization}, the defining
relations agree with (Y0)--(Y6), including
(Y4$'$)--(Y5$'$), of \cite[Sections~1.3--1.4]{Tsy}.
In particular, the current relations agree since
\[
\frac{\prod_{i=1}^3(s+h_i)}{\prod_{i=1}^3(s-h_i)}-1
=\frac{-2\varepsilon t}
{s^3+(\varepsilon-t^2)s+\varepsilon t}
=\varepsilon t V(s).
\]

For the geometric construction, we retain the presentation before
inverting $\varepsilon$. The parameters $t$ and $\varepsilon$ act by
multiplication by $K_X$ and $e_X$, respectively, on the $H^*(X)$ input.
In the next subsection, we define $A*B$ and $[A,B]_{\mathrm{red}}$
for the operator families considered here and prove
\[
(A*B-B*A)(\gamma)=[A,B]_{\mathrm{red}}(e_X\gamma).
\]
This realizes the first relation in \eqref{eq:divided-algebra-rules}
without dividing by $e_X$. We also verify the remaining rules for
these operations before checking the defining Yangian relations.

\subsubsection{Products and reduced commutators}
We now construct the multiplication and bracket corresponding to those in the preceding subsection. Recall that our operators are maps
\[
A:\mathbb H_X\otimes H^*(X)\longrightarrow\mathbb H_X.
\]We use the small diagonal to define the multiplication \(A*B\), and the transfer property to define the bracket \([A,B]_{\mathrm{red}}\).
 For two such families, define
\begin{equation}\label{eq:star-red}
\begin{aligned}
(A*B)(\gamma)&=(AB)(\tau_{2*}\gamma)
=\sum_\ell A(\gamma_\ell')B(\gamma_\ell''),\\
[A,B]_{\mathrm{red}}(\gamma)&=[A(1_X),B(\gamma)],
\end{aligned}
\end{equation}
where $\tau_{2*}\gamma=\sum_\ell\gamma_\ell'\otimes\gamma_\ell''$. The reduced bracket is used for families with the transfer property
\[
[A(\alpha),B(\beta)]=[A,B]_{\mathrm{red}}(\alpha\beta).
\]
The transfer property means that
$[A(\alpha),B(\beta)]$ depends on $\alpha$ and $\beta$
only through their product $\alpha\beta$.
Thus the commutator is completely determined by
setting $\alpha=1_X$, which gives
$[A,B]_{\mathrm{red}}$.
For $[E_a, E_b]$, this property was established
in Lemma~\ref{lem:EE-transfer}.

The product $*$ is associative, because $(\tau_{2*}\otimes\Id)\tau_{2*}=(\Id\otimes\tau_{2*})\tau_{2*}=\tau_{3*}$. Its unit is the family $\mathbf1:\gamma\mapsto(\int_X\gamma)\Id$. For an even class $\eta\in H^*(X)$, its action on a family is given by $(\eta A)(\gamma)=A(\eta\gamma)$. Graded cocommutativity and diagonal self-intersection give
\begin{equation}\label{eq:geometric-divided-rules}
\begin{split}
(A*B-B*A)(\gamma)
&=\sum_\ell[A(\gamma_\ell'),B(\gamma_\ell'')]\\
&=[A,B]_{\mathrm{red}}\left(\sum_\ell\gamma_\ell'\gamma_\ell''\right)
=[A,B]_{\mathrm{red}}(e_X\gamma).
\end{split}
\end{equation}
Thus $A*B-B*A=e_X[A,B]_{\mathrm{red}}$ in the parameter convention just stated; $e_X$ acts on the $H^*(X)$ input, not on the Hilbert-scheme output.

The reduced bracket satisfies Jacobi identity and the Leibniz rule. For Jacobi identity, evaluate the cyclic sum on $\gamma$. By the transfer property its three terms become
\[
[A(1_X),[B(1_X),C(\gamma)]]
+[B(1_X),[C(\gamma),A(1_X)]]
+[C(\gamma),[A(1_X),B(1_X)]],
\]
which is zero. For Leibniz rule, apply the even operator $A(1_X)$ to the product in \eqref{eq:star-red}:
\[
\begin{split}
[A,B*C]_{\mathrm{red}}(\gamma)
&=\sum_\ell[A(1_X),B(\gamma_\ell')C(\gamma_\ell'')]\\
&=\bigl([A,B]_{\mathrm{red}}*C+B*[A,C]_{\mathrm{red}}\bigr)(\gamma).
\end{split}
\]

We now check that the transfer property is preserved
when we combine these operators.
The relations proved above show that $E_k$, $F_k$,
and the coefficients of $\mathscr C_n(z)$ satisfy
the transfer property.

Since the small diagonal satisfies
\[
\tau_{2*}(\alpha\beta)
=(\alpha\otimes 1_X)\cup\tau_{2*}\beta,
\]
the Leibniz rule shows that the transfer property
is preserved under the multiplication $*$.
The Jacobi identity gives the same conclusion
for the reduced bracket.
Thus the operators obtained from $E$, $F$, and
$\mathscr C$ by taking products and reduced brackets
still satisfy the transfer property.

\subsubsection{Completion of the proof}
\begin{proof}[Proof of the action statement in Theorem~\ref{thm:main}]
By \eqref{eq:geometric-divided-rules} and the Jacobi identity and
Leibniz rule established above, the multiplication
$*$ and the reduced bracket satisfy
\eqref{eq:divided-algebra-rules} under
$t\mapsto K_X$ and $\varepsilon\mapsto e_X$.

It remains to verify the defining relations.
The commutativity of the Cartan operators and
\eqref{eq:main-EF} give \eqref{eq:abstract-mixed}.
Setting $\alpha=1_X$ in \eqref{eq:divided-HE}
gives \eqref{eq:abstract-CE}, using
\eqref{eq:V-surface} and the definitions of
$*$ and the reduced bracket.
The corresponding lowering relation gives
\eqref{eq:abstract-FC}.

By the transfer property,
Theorem~\ref{thm:EE-independent} gives
\eqref{eq:abstract-EE} and its negative counterpart.
Similarly, Theorem~\ref{thm:Serre} gives
\eqref{eq:abstract-Serre}.

Thus the assignments
\[
e_k\mapsto E_k,\qquad
f_k\mapsto F_k,\qquad
c_k\mapsto\mathcal C_k
\]
satisfy all the defining relations.
Finally, $\kappa_{0,n}=1$ and $\kappa_{1,n}=0$
give $c_0=1$ and $c_1=0$, where the unit acts by
$\gamma\mapsto(\int_X\gamma)\operatorname{Id}$.
This completes the proof.
\end{proof}

\appendix
\section{A direct proof of the two coefficient identities}\label{app:coefficient-identities}
In this appendix, we give a direct combinatorial proof of Proposition~\ref{prop:two-identities}. We first express the coefficients obtained in Section~\ref{sec:recursions} in terms of two sums over permutations. These expressions are verified using the coefficient recursions and finite reconstruction, independently of Proposition~\ref{prop:two-identities}. The proof then reduces to counting the ways of inserting one or two additional labels into permutations. Throughout the appendix, all weights are integers. We consider total weight one for the derivative coefficients and total weight zero for the Chern character coefficients.
\subsection{Shorter formulas for Chern character coefficients}
Let $\mathcal A=(p_i)_{i\in I}$ be a labelled family, with $|I|=m$, and put $x_+=\max(x,0)$. For a permutation $\sigma$ of $I$, consider its disjoint cycle decomposition. Let its cycles be regarded as labelled subsets and set
\[
u(\sigma)=\sum_{B\text{ a cycle of }\sigma}(p_B)_+,
\qquad
S_j(\mathcal A)=\sum_{\sigma\in\mathfrak S_I}u(\sigma)^j,
\quad j=1,2.
\]
We set $S_1(\varnothing)=S_2(\varnothing)=0$. These sums can be evaluated without enumerating permutations:
\begin{align}
S_1(\mathcal A)
&=\sum_{n=1}^{m}(n-1)!(m-n)!
\sum_{\substack{B\subseteq I\\|B|=n}}(p_B)_+,
\label{eq:direct-S1}\\
S_2(\mathcal A)
&=\sum_{n=1}^{m}(n-1)!(m-n)!
\sum_{\substack{B\subseteq I\\|B|=n}}(p_B)_+^2\notag\\
&\quad+\sum_{u=1}^{m-1}\sum_{v=1}^{m-u}
(u-1)!(v-1)!(m-u-v)!
\sum_{\substack{B,C\subseteq I\\B\cap C=\varnothing\\|B|=u,\ |C|=v}}
(p_B)_+(p_C)_+.
\label{eq:direct-S2}
\end{align}
Indeed, a prescribed cycle with size $n$ forms $(n-1)!$ different cycles and the remaining labels become $(m-n)!$ permutations. The second moment has either one distinguished cycle or two ordered distinguished cycles, giving \eqref{eq:direct-S2}. 

Now we deduce another useful identity. Since $$\sum_{B\text{ a cycle of }\sigma}(p_B)_+ - \sum_{B\text{ a cycle of }\sigma}(-p_B)_+ = s(\mathcal A),$$ there is 
\[
u(\sigma)\sum_{B\text{ a cycle of }\sigma}(-p_B)_+ = u(\sigma)^2 - u(\sigma)s(\mathcal A).
\]
Summing over all permutations gives 
\[
S_2(\mathcal A) - s(\mathcal A)S_1(\mathcal A) 
= \sum_{\sigma\in\mathfrak S_I}u(\sigma)\sum_{B\text{ a cycle of }\sigma}(-p_B)_+.
\]
Expanding the right-hand side by the number of cycles and their sizes gives
\begin{equation}\label{eq:direct-S2-opposite}
    S_2(\mathcal A) - s(\mathcal A)S_1(\mathcal A) =\sum_{u=1}^{m-1}\sum_{v=1}^{m-u}
    (u-1)!(v-1)!(m-u-v)!
    \sum_{\substack{B,C\subseteq I,\ B\cap C=\varnothing\\
    |B|=u,\ |C|=v}}
    (p_B)_+(-p_C)_+.
\end{equation}
We have the following counting formulas for the sums over labelled subsets.
\begin{lemma}\label{lem:direct-deletion}
If $s(\mathcal A)=0$ and $m\geq2$, then for each $i\in I$,
\begin{equation}\label{eq:direct-deletion}
\frac{S_1(\mathcal A\backslash\{p_i\})}{(m-1)!}
=\frac{S_1(\mathcal A)}{m!}-\frac{p_i}{2}.
\end{equation}
Here the deletion removes the specified label. Moreover, for $1\leq r\leq m$,
\begin{equation}\label{eq:direct-marked-cycles}
\sum_{\pi\in\Pi_r(I)}\prod_{B\in\pi}(|B|-1)!p_B
=\frac{(m-1)!}{(r-1)!}
\sum_{\substack{J\subseteq I\\|J|=r}}\prod_{i\in J}p_i.
\end{equation}
In particular, if $m\geq3$ and $s(\mathcal A)=0$, then
\begin{equation}\label{eq:direct-triple-cycles}
\sum_{\pi\in\Pi_3(I)}\prod_{B\in\pi}(|B|-1)!p_B
=\frac{(m-1)!}{6}\sum_{i\in I}p_i^3.
\end{equation}
Each block in \eqref{eq:direct-marked-cycles} has size between $1$ and $m-r+1$.
\end{lemma}
\begin{proof}
In \eqref{eq:direct-S1}, decompose a subset of $I$ containing $i$ as $B\sqcup\{i\}$, then we get
\[
S_1(\mathcal A)
={}
\sum_{n=1}^{m-1}
(n-1)!(m-n)!
\sum_{\substack{B\subseteq I\backslash\{i\}\\|B|=n}}
(p_B)_+ +
\sum_{n=0}^{m-1}
n!(m-n-1)!
\sum_{\substack{B\subseteq I\backslash\{i\}\\|B|=n}}
(p_i+p_B)_+.
\]
Subtraction gives
\[
\begin{gathered}
S_1(\mathcal A)
-mS_1(\mathcal A\backslash\{p_i\})
={}
\sum_{n=0}^{m-1}
n!(m-n-1)!
\sum_{\substack{B\subseteq I\backslash\{i\}\\|B|=n}}
(p_i+p_B)_+
\\
-
\sum_{n=1}^{m-1}
n!(m-n-1)!
\sum_{\substack{B\subseteq I\backslash\{i\}\\|B|=n}}
(p_B)_+.
\end{gathered}
\]
By taking $B' = I\backslash (B\sqcup\{i\})$, we have $|B'|=m-n-1$, then $p_{B'}=s(\mathcal{A})-p_B-p_i = -(p_B+p_i)$. The formula $(-x)_+-x_+=-x$ implies
\[
S_1(\mathcal A)-mS_1(\mathcal A\backslash\{p_i\})
=-\sum_{n=1}^{m-1}(n-1)!(m-n)!
\sum_{\substack{B\subseteq I\backslash\{i\}\\|B|=n}}p_B
=\frac{m!p_i}{2}.
\]
For the last equality, the inner sum is
$-\binom{m-2}{n-1}p_i$, and
\(\sum_{n=1}^{m-1}
(n-1)!(m-n)!\binom{m-2}{n-1}
=\frac{m!}{2}\). This proves \eqref{eq:direct-deletion}.

For \eqref{eq:direct-marked-cycles}, selecting a label $i_B$ from each block $B\in\pi$, there is the expansion
\[
\prod_{B\in\pi}p_B
=
\sum_{\substack{i_B\in B\\B\in\pi}}
\prod_{B\in\pi}p_{i_B}.
\]
The chosen labels $i_B$ are distinct, so they form a subset $J\subseteq I$ of size $r$. Conversely, for a fixed subset $J\subseteq I$ of size $r$, $\prod_{i\in J}p_i$ appears in the expansion of $\prod_{B\in\pi}p_B$ if and only if each block $B\in\pi$ contains exactly one label from $J$. Left hand side of \eqref{eq:direct-marked-cycles} is therefore
\[
\sum_{\substack{J\subseteq I\\|J|=r}}c_J\prod_{i\in J}p_i,
\]
where 
\[
c_J = \sum_{\substack{\pi\in\Pi_r(I)\\|B\cap J|=1,\ B\in\pi}}\prod_{B\in\pi}(|B|-1)!.
\]
Since $(|B|-1)!$ counts the number of cycles supported on $B$, there is 
\[
c_J
=
\#\left\{
\sigma\in\mathfrak S_I:
\text{every cycle of }\sigma
\text{ contains exactly one element of }J
\right\}.
\]
This number is exactly $\frac{(m-1)!}{(r-1)!}$ by direct counting. When $r=3$ and $s(\mathcal A)=0$, \eqref{eq:direct-triple-cycles} follows from \eqref{eq:direct-marked-cycles} and the identity
\[
\sum_{i<j<k}p_ip_jp_k = \frac13\sum_i p_i^3.
\]
\end{proof}

\begin{lemma}\label{lem:direct-merge-moments}
For $s(\mathcal A)=0$ and $m\geq2$, the merge operator of \eqref{eq:merge-def} satisfies
\begin{align}
\mathfrak U S_1(\mathcal A)
&=\frac{(m-1)!}{2}\Omega(\mathcal A),
\label{eq:direct-merge-S1}\\
\mathfrak U S_2(\mathcal A)
&=\frac{\Omega(\mathcal A)}{m}S_1(\mathcal A)
-\frac{(m-1)!}{6}\sum_{i\in I}p_i^3,
\label{eq:direct-merge-S2}
\end{align}
where $\Omega(\mathcal A)=\sum_i p_i(1-|p_i|)=-\sum_i p_i|p_i|$.
\end{lemma}
\begin{proof}
By definition, $\mathfrak{U}S_1(\mathcal{A}) = \sum_{\{i,j\}\subseteq I;i\neq j}(p_i+p_j)S_1(\mathcal A_{i,j})$. For a subset $B\subset I$ distinguished by whether containing $i,j$ or not, $S_1(\mathcal{A}_{i,j})$ can be rewritten as
\[
\begin{aligned}
S_1(\mathcal A_{i,j})
={}&
\sum_{n=1}^{m-2}
(n-1)!(m-n-1)!
\sum_{\substack{B\subseteq I\setminus\{i,j\}\\|B|=n}}
(p_B)_+
\\
&+
\sum_{n=2}^{m}
(n-2)!(m-n)!
\sum_{\substack{B\subseteq I\\|B|=n\\i,j\in B}}
(p_B)_+.
\end{aligned}
\]
For any subset $C\subseteq I$ with $|C|\geq2$, the following identity is direct:
\[
\sum_{\substack{\{i,j\}\subseteq C\\i\neq j}}(p_i+p_j)
=(|C|-1)p_C.
\]
Here the sum runs over unordered pairs of distinct elements.
$s(\mathcal{A})=0$ gives $p_{I\backslash B} = -p_B$. The merged sum is thus:
\[
\begin{aligned}
\mathfrak{U}S_1(\mathcal A)
={}&
\sum_{n=1}^{m-2}
(n-1)!(m-n-1)!
\sum_{\substack{B\subseteq I\\|B|=n}}
(p_B)_+
\sum_{\{i,j\}\subseteq I\setminus B}
(p_i+p_j)
\\
&+
\sum_{n=2}^{m}
(n-2)!(m-n)!
\sum_{\substack{B\subseteq I\\|B|=n}}
(p_B)_+
\sum_{\{i,j\}\subseteq B}
(p_i+p_j)
\\
=&
-\sum_{n=1}^{m-2}
(n-1)!(m-n-1)!(m-n-1)
\sum_{\substack{B\subseteq I\\|B|=n}}
p_B(p_B)_+
\\
&+
\sum_{n=2}^m (n-1)!(m-n)!\sum_{\substack{B\subseteq I\\|B|=n}}p_B(p_B)_+
\\
={}&
\sum_{n=2}^{m-1}
(n-1)!(m-n-1)!
\sum_{\substack{B\subseteq I\\|B|=n}}
p_B(p_B)_+
-(m-2)(m-2)!
\sum_{i\in I}p_i(p_i)_+.
\end{aligned}
\]
By comparing with the expansion of $-\sum_i p_iS_1(\mathcal A\backslash\{p_i\})$ there is 
\[
\mathfrak U S_1(\mathcal A)
=-\sum_i p_iS_1(\mathcal A\backslash\{p_i\})
-(m-1)!\sum_i p_i(p_i)_+.
\]
By \eqref{eq:direct-deletion}, the first term is
$(m-1)!\sum_i p_i^2/2$. Using
$2p_i(p_i)_+=p_i^2+p_i|p_i|$ and
$\Omega(\mathcal A)=-\sum_i p_i|p_i|$ proves \eqref{eq:direct-merge-S1}.

For the proof of \eqref{eq:direct-merge-S2}, we split it into two steps:
First we show
\[\mathfrak U S_2
=
-\sum_i p_i\Bigl(S_2(\mathcal A\backslash\{p_i\})
+2(p_i)_+S_1(\mathcal A\backslash\{p_i\})
+(m-1)!(p_i)_+^2\Bigr)\]
and
\[\sum_i p_iS_2(\mathcal A\backslash\{p_i\})
+\sum_i p_i^2S_1(\mathcal A\backslash\{p_i\})
=
-2\sum_{\pi\in\Pi_3(I)}
\prod_{B\in\pi}(|B|-1)!p_B,\]
thus
\[
\mathfrak U S_2(\mathcal A)={}
-\sum_i p_i|p_i|S_1(\mathcal A\backslash\{p_i\})
-(m-1)!\sum_i(p_i)_+^3+2\sum_{\pi\in\Pi_3(I)}\prod_{B\in\pi}(|B|-1)!p_B.
\]
Then apply \eqref{eq:direct-deletion} and \eqref{eq:direct-triple-cycles}, and use
$2\sum_i(p_i)_+^3=\sum_i p_i^3+\sum_i|p_i|^3$.
The absolute cubes cancel and leave \eqref{eq:direct-merge-S2}. For $m=2$ there is no three-block partition and $\sum_i p_i^3=0$, so the same calculation applies.

Now we consider the two identities in the first step. By definition, 
\[\mathfrak U S_2(\mathcal A)
=
\sum_{\{i,j\}}
(p_i+p_j)
\sum_{\tau\in\mathfrak S_{I_{i,j}}}
u(\tau)^2,
\]
where $i,j$ are merged into a new label $\diamondsuit$. Its weight is $p_\diamondsuit = p_i+p_j$. A useful observation is for $\{i,j\}$ and $\tau\in\mathfrak S_{I_{i,j}}$, we can substitute $\diamondsuit$ in the cycle containing $\diamondsuit$ by the ordered pair $(i,j)$ or $(j,i)$. In this way, we recover a permutation $\sigma\in\mathfrak S_I$ with the same cycle sums $u(\tau) = u(\sigma)$, and moreover these data, $\{i,j\}$, $\tau$ and a choice of order of $\{i,j\}$, correspond to $\sigma\in\mathfrak S_I$ together with a label $a\in I$ such that $\sigma(a)\neq a$. Thus
\[
\begin{aligned}
\mathfrak U S_2(\mathcal A)
&=
\frac12\sum_{\sigma\in\mathfrak S_I}
u(\sigma)^2
\sum_{\substack{a\in I\\\sigma(a)\ne a}}
\bigl(p_a+p_{\sigma(a)}\bigr)
\\
&=
-\sum_{\sigma\in\mathfrak S_I}
u(\sigma)^2
\sum_{\substack{i\in I\\\sigma(i)=i}}p_i.
\end{aligned}
\]
Exchanging the summation order gives the first identity. For the second identity we exploit \eqref{eq:direct-S2-opposite} to get
\[
\begin{aligned}
&\sum_i p_iS_2(\mathcal A\backslash\{p_i\})
+\sum_i p_i^2S_1(\mathcal A\backslash\{p_i\})
\\
=&
\sum_i p_i\bigl(S_2(\mathcal A\backslash\{p_i\}) + p_iS_1(\mathcal A\backslash\{p_i\})\bigr)
\\
=&
\sum_{i\in I}p_i
\sum_{u=1}^{m-2}
\sum_{v=1}^{m-1-u}
(u-1)!(v-1)!(m-u-v-1)!
\\
&\qquad\cdot
\sum_{\substack{
B,C\subseteq I\setminus\{i\}\\
B\cap C=\varnothing\\
|B|=u,\ |C|=v}}
(p_B)_+(-p_C)_+
\\
=&
\sum_{\substack{
B,C,D\ne\varnothing\\
B\sqcup C\sqcup D=I\\
(B,C)\text{ ordered}
}}
(|B|-1)!(|C|-1)!(|D|-1)!
\,p_D(p_B)_+(-p_C)_+.
\end{aligned}
\]
The last step is due to exchanging the summation order. For integers $q_1,q_2,q_3$ with $q_1+q_2+q_3=0$, there is
\[
\sum_{\substack{a,b,c\text{ distinct}\\
\{a,b,c\}=\{1,2,3\}}}
q_c(q_a)_+(-q_b)_+
=
-2q_1q_2q_3.
\]
By assigning $\{q_1,q_2,q_3\}=\{p_B,p_C,p_D\}$ we get the identity we want. The last term is exactly $-2\sum_{\pi\in\Pi_3(I)}\prod_{B\in\pi}(|B|-1)!p_B$.
\end{proof}

For a family $\mathcal A$ with $s(\mathcal A)=1$, let $\mathcal A^{(i)}$ replace the weight $p_i$ by $p_i-1$. It has total weight zero. In $\mathcal A\sqcup\{-1\}$ the appended entry is a new label, even if $-1$ already occurs.

\begin{lemma}\label{lem:direct-unit-shift}
For integer weights of total one and $m\geq1$,
\begin{align}
\sum_i(1-p_i)S_1(\mathcal A^{(i)})
&=mS_1(\mathcal A)-\frac{m+1}{2}m!,
\label{eq:direct-shift-S1}\\
\sum_i(1-p_i)\bigl(S_2(\mathcal A^{(i)})-S_1(\mathcal A^{(i)})\bigr)
&=S_2(\mathcal A\sqcup\{-1\})-S_1(\mathcal A\sqcup\{-1\}).
\label{eq:direct-shift-S2}
\end{align}
\end{lemma}
\begin{proof}
For fixed \(\sigma\), let
\[
L_+=L_+(\sigma)
:=
\sum_{\substack{C\text{ a cycle of }\sigma\\ p_C>0}}|C|,
\]the number of labels lying in cycles of strictly positive total weight.
For \(i\in I\), let \(C_i\) be the cycle of \(\sigma\) containing \(i\). Passing from \(\mathcal A\) to \(\mathcal A^{(i)}\) changes only the weight of \(C_i\), from \(p_{C_i}\) to \(p_{C_i}-1\). Since all weights are integers,
\[
(x-1)_+=x_+-\mathbf 1_{x>0}.
\]Hence
\[
u(\sigma;\mathcal A^{(i)})
=
u(\sigma)-\mathbf 1_{p_{C_i}>0}.
\]Now
\[
\sum_{i\in I}(1-p_i)=m-1,
\]because \(s(\mathcal A)=1\). Moreover,
\[
\begin{aligned}
\sum_i(1-p_i)\mathbf 1_{p_{C_i}>0}
&=
\sum_{\substack{C\text{ a cycle of }\sigma\\p_C>0}}
\sum_{i\in C}(1-p_i)\\
&=
\sum_{C:p_C>0}(|C|-p_C)\\
&=
L_+-u(\sigma),
\end{aligned}
\]since \(p_C>0\) implies \((p_C)_+=p_C\). Therefore
\[
\begin{aligned}
\sum_i(1-p_i)u(\sigma;\mathcal A^{(i)})
&=
\sum_i(1-p_i)(u(\sigma)-\mathbf 1_{p_{C_i}>0})\\
&=
(m-1)u(\sigma)-(L_+-u(\sigma))\\
&=
mu(\sigma)-L_+.
\end{aligned}
\]Summing over \(\sigma\in\mathfrak S_I\) gives
\[
\sum_i(1-p_i)S_1(\mathcal A^{(i)})
=
mS_1(\mathcal A)
-
\sum_{\sigma\in\mathfrak S_I}L_+(\sigma).
\]
It remains to calculate the last sum. Fix a nonempty subset
\[
B\subseteq I,\qquad |B|=b.
\]If \(p_B>0\), then whenever \(B\) is the support of a cycle of \(\sigma\), that cycle contributes \(b\) to \(L_+(\sigma)\). There are
\[
(b-1)!(m-b)!
\]permutations for which \(B\) is the support of one cycle: there are \((b-1)!\) cyclic orders on \(B\), while the complement is arbitrarily permuted. Hence
\[
\sum_{\sigma}L_+(\sigma)
=
\sum_{b=1}^m
b!(m-b)!
\sum_{\substack{B\subseteq I\\|B|=b}}
\mathbf 1_{p_B>0}.
\]For every proper nonempty \(B\subsetneq I\),
\[
p_{I\setminus B}=1-p_B.
\]
Since \(p_B\in\mathbb Z\), exactly one of \(p_B\) and \(1-p_B\) is strictly positive. Thus 
\[
\sum_\sigma L_+(\sigma) - m! = \sum_{b=1}^{m-1} b!(m-b)!\sum_{\substack{B\subseteq I\\|B|=b}}\mathbf 1_{p_B\leq0}.
\]
$\mathbf 1_{p_B\leq0} + \mathbf 1_{p_B>0} = 1$, so
\[
\sum_{\sigma} L_+(\sigma) = m! + \frac12\sum_{b=1}^{m-1} b!(m-b)!\binom{m}{b} = \frac{m+1}{2}m!.
\]
Substituting into the formula above proves
\[
\sum_i(1-p_i)S_1(\mathcal A^{(i)})
=
mS_1(\mathcal A)-\frac{m+1}{2}m!.
\]
The same calculation applied to $u(u-1)$ gives
\[
\sum_i(1-p_i)u(\sigma;\mathcal A^{(i)})
\bigl(u(\sigma;\mathcal A^{(i)})-1\bigr)
=(m+1)u(\sigma)(u(\sigma)-1)-2L_+(\sigma)(u(\sigma)-1).
\]
On the other hand, append a new label $\diamondsuit$ with weight $-1$ to $I$. For any $\sigma'\in\mathfrak S_{I\sqcup \{\diamondsuit\}}$, $\diamondsuit$ is either a singleton cycle or is contained in a cycle of length at least two. $\sigma'$ reduces to a permutation $\sigma\in\mathfrak S_I$ by removing $\diamondsuit$ from its cycle. If $\diamondsuit$ is a singleton cycle, then
$u(\sigma')=u(\sigma)$.
If it is inserted into a cycle $C$ of $\sigma$, then
$u(\sigma')=u(\sigma)-1$ when $p_C>0$, and
$u(\sigma')=u(\sigma)$ otherwise. For a cycle $C$ of $\sigma$, there are $|C|$ ways to insert $\diamondsuit$ into that cycle. Hence
\[
\begin{gathered}
\sum_{\sigma'\in\mathfrak S_{I\sqcup \{\diamondsuit\}}}u(\sigma';\mathcal A\sqcup\{-1\})\bigl(u(\sigma';\mathcal A\sqcup\{-1\})-1\bigr) \\
= \sum_{\sigma\in\mathfrak S_I}\bigl((m+1-L_+(\sigma))u(\sigma)(u(\sigma)-1) + L_+(\sigma)(u(\sigma)-1)(u(\sigma)-2)\bigr) \\
= \sum_{\sigma\in\mathfrak S_I}\bigl((m+1)u(\sigma)(u(\sigma)-1)-2L_+(\sigma)(u(\sigma)-1)\bigr).
\end{gathered}
\]
This proves \eqref{eq:direct-shift-S2}.
\end{proof}

We use the local abbreviation
$g_\eta^{(k)}(\mathcal A)=g_{\eta,(-p_i)_{i\in I}}^{(k)}$,
just as the notation $f_\eta^{(k)}(\mathcal A)$, where $\eta = 1,e_X,K_X,K_X^2$.

\begin{proposition}\label{prop:direct-four-formulas}
For integer weights with $s(\mathcal A)=1$ and $|I|=k\geq1$,
\begin{align}
f_{K_X}^{(k)}(\mathcal A)
&=S_1(\mathcal A)-k!,
\label{eq:direct-fK}\\
f_{K_X^2}^{(k+1)}(\mathcal A)
&=\frac{S_2(\mathcal A\sqcup\{-1\})-S_1(\mathcal A\sqcup\{-1\})}{2}
+\frac{(k+1)!}{24}\left(1-\sum_i p_i^2\right).
\label{eq:direct-fK2}
\end{align}
For $s(\mathcal A)=0$ and $|I|=k\geq2$,
\begin{align}
g_{K_X}^{(k-1)}(\mathcal A)
&=\frac12-\frac{S_1(\mathcal A)}{k!},
\label{eq:direct-gK}\\
g_{K_X^2}^{(k)}(\mathcal A)
&=-\frac1{12}+\frac{S_1(\mathcal A)-S_2(\mathcal A)}{2k!}
+\frac1{24}\sum_i p_i^2.
\label{eq:direct-gK2}
\end{align}
The coefficient statements concern nonzero modes. The finite sums on the right are also defined when a weight is zero and provide the extensions used in the calculation below.
\end{proposition}
\begin{proof}
First we show \eqref{eq:direct-fK}. We rewrite \eqref{eq:fK-polynomial} at $s(\mathcal A)=1$. For $1\leq b\leq m-1$, let
\[
P_{m,b}(x)=\sum_{a=b+1}^{m}
\frac{(-1)^{a-b-1}m!}{a(a-1)}
\binom{m-b-1}{a-b-1}x^a.
\]
Its second derivative is $m!x^{b-1}(1-x)^{m-b-1}$, and its value and first derivative vanish at zero. There is 
\[
\frac{d^2}{dx^2}
\left(P_{m,b}(x)-P_{m,m-b}(1-x)\right)=0.
\]
Hence the difference is linear:
\[
P_{m,b}(x)-P_{m,m-b}(1-x)
=(b-1)!(m-b-1)!(mx-b).
\]
The constants follow by integrating the second derivative $P''_{m,m-b}(x)$ on $[0,1]$, while this equality holds for $x\in \mathbb{R}$. 
Comparing with \eqref{eq:fK-polynomial} we get
\[f_{K_X}^{(k)} = -\frac{(k-1)(k-1)!}{2} +\frac12 \sum_{\varnothing\neq B\subsetneq I} \operatorname{sgn}(p_B) P_{k,|B|}(p_B).\]

For an integer $x$, either $x\leq0$ or $x\geq1$, thus exactly one of \(p_B\) and \(p_{I\setminus B}\) is positive. Pairing $B$ and $I\backslash B$ in \eqref{eq:fK-polynomial} thus gives
\begin{align*}
f_{K_X}^{(k)}(\mathcal A)
&=-\frac{(k-1)(k-1)!}{2}
+\frac14\sum_{b=1}^{k-1}(b-1)!(k-b-1)!
\sum_{\substack{B\subseteq I\\|B|=b}}|kp_B-b|\\
&=-\frac{k!}{2}
+\frac12\sum_{b=1}^{k}(b-1)!(k-b)!
\sum_{\substack{B\subseteq I\\|B|=b}}|p_B|.
\end{align*}
We used $|kx-b|=(k-b)|x|+b|1-x|$, valid for $x\in\mathbb Z$. Also
$\sum_b(b-1)!(k-b)!\sum_{|B|=b}p_B=k!$.
Since $2x_+=|x|+x$, we prove \eqref{eq:direct-fK}. The $k=1$ case can be checked directly.


We next show \eqref{eq:direct-gK}. From 
\(g_{\epsilon,\mu\sqcup\{1\}}^{(k)}
=
-c_{k,1,\epsilon}(\mu)
=
-\frac1{k!}f_{\epsilon,\mu}^{(k)}\)
and
\(
[\mathfrak G_k(\gamma),\mathfrak a_{-1}'(1_X)]
=
\frac1{k!}\mathfrak a_{-1}^{(k+1)}(\gamma)
\)
we can read out the following two relations for the coefficients \(g_\epsilon^{(k)}\) and \(f_\epsilon^{(k)}\)
\begin{align}
g_\epsilon^{(k)}(\mathcal B\sqcup\{-1\})
&=-\frac1{k!}f_\epsilon^{(k)}(\mathcal B),
\label{eq:direct-reconstruction-initial}\\
f_\epsilon^{(k+1)}(\mathcal A)
&=-k!\sum_i(1-p_i)g_\epsilon^{(k)}(\mathcal A^{(i)}).
\label{eq:direct-reconstruction-shift}
\end{align}
Here $s(\mathcal B)=s(\mathcal A)=1$, with the respectively suitable length. These two formulas determine $g_\epsilon^{(k)}$ by induction. For $g_\epsilon^{(k)}(\mathcal C)$ with $\mathcal C = \{q_i\}_{i\in I}$, if $-1\in\mathcal C$ then \eqref{eq:direct-reconstruction-initial} gives $g_\epsilon^{(k)}(\mathcal C) = -\frac{1}{k!}f_{\epsilon}^{(k)}(\mathcal C\backslash \{-1\})$. Suppose we know $g_\epsilon^{(k)}(\mathcal C)$ for all $\mathcal C$ containing $-n$ and $-(n+1)\in\mathcal C'$. Let $p_{i} = -(n+1)$ for some $i\in I$. Define $\mathcal A' = \mathcal C'\backslash\{-(n+1)\}\sqcup\{-n\}$, then $s(\mathcal{A}') = 1, \mathcal C' = \mathcal{A'}^{(i)}$. By \eqref{eq:direct-reconstruction-shift} we have
\[
\begin{gathered}
    f_\epsilon^{(k+1)}(\mathcal A')
=-k!\sum_{j\neq i}(1-p_j)g_\epsilon^{(k)}(\mathcal{A'}^{(j)})-k!(1+n)g_\epsilon^{(k)}(\mathcal C')\\
g_\epsilon^{(k)}(\mathcal C') = -\frac{1}{k!(1+n)}f_\epsilon^{(k+1)}(\mathcal A')-\frac{1}{1+n}\sum_{j\neq i}(1-p_j)g_\epsilon^{(k)}(\mathcal{A'}^{(j)}),
\end{gathered}
\]
while $-n\in\mathcal{A'}^{(j)}$ for $j\neq i$, so $g_\epsilon^{(k)}(\mathcal{A'}^{(j)})$ is known by the induction hypothesis. This proves that $g_\epsilon^{(k)}(\mathcal C')$ is determined by the two reconstruction formulas.

The expression on the right of \eqref{eq:direct-gK} satisfies \eqref{eq:direct-reconstruction-initial} by Lemma~\ref{lem:direct-deletion} and \eqref{eq:direct-fK}. It satisfies \eqref{eq:direct-reconstruction-shift} by \eqref{eq:direct-shift-S1} and $\sum_i(1-p_i)=m-1$ with $m=k+1$. Therefore it is the coefficient given by \eqref{eq:g-reconstruction}, proving \eqref{eq:direct-gK}.


For the term involving $K_X^2$, temporarily write $h_m(\mathcal A)$ for the right side of \eqref{eq:direct-gK2}, on $\mathcal A$ of length $m$ and $s(\mathcal A) = 0$. Lemma~\ref{lem:direct-merge-moments} and a direct computation gives
\begin{equation}\label{eq:direct-h-merge}
\sum_{\{i,j\}\in\binom I2}(p_i+p_j)
 h_{m-1}(\mathcal A_{i,j})
=\frac{\Omega(\mathcal A)}2
\left(\frac12-\frac{S_1(\mathcal A)}{m!}\right),\qquad m\geq3.
 \end{equation}

When $s(\mathcal A)=1$, \eqref{eq:direct-shift-S2} also gives
\begin{equation}\label{eq:direct-h-shift}
-m!\sum_i(1-p_i)h_m(\mathcal A^{(i)})
=-(m+1)!h_{m+1}(\mathcal A\sqcup\{-1\}).
\end{equation}
Both sides of \eqref{eq:direct-h-shift} are then exactly the right side of \eqref{eq:direct-fK2}. Call this expression $b_m(\mathcal A)$. It has $b_1(\{1\})=0$. Apply \eqref{eq:direct-h-merge} to the total-zero family $\mathcal A\sqcup\{-1\}$ and separate pairs involving the new label. By \eqref{eq:direct-reconstruction-initial} for the already proved coefficient of $K_X$ and by \eqref{eq:direct-h-shift},
\begin{align}
\mathfrak U b_{m-1}(\mathcal A)
&=-m!\left(\frac{\Omega(\mathcal A)}2
 g_{K_X}^{(m)}(\mathcal A\sqcup\{-1\})
 -\sum_i(p_i-1)h_m(\mathcal A^{(i)})\right)\notag\\
&=\frac{\Omega(\mathcal A)}2f_{K_X}^{(m)}(\mathcal A)+b_m(\mathcal A).
\label{eq:direct-b-recursion}
\end{align}
We used $\Omega(\{-1\})=0$. Thus $b_m$ satisfies exactly \eqref{eq:rec-K2} when $s(\mathcal A)=1$, with the same initial value. Uniqueness proves \eqref{eq:direct-fK2}.

Finally, $b_m$ and $h_m$ satisfy \eqref{eq:direct-reconstruction-initial} and \eqref{eq:direct-reconstruction-shift} by definition of $b_m$ and \eqref{eq:direct-h-shift}.
The same finite-reconstruction uniqueness proves \eqref{eq:direct-gK2}. 
\end{proof}

\subsection{Proof of Proposition~\ref{prop:two-identities}}\label{subsec:proof-two-identities} We use the formulas of Proposition~\ref{prop:direct-four-formulas} to prove the two identities in Proposition~\ref{prop:two-identities}. 
\subsubsection{Proof of \eqref{eq:identity-K}}
Let $\lambda=(-p_i)_{i\in I}$ with $m=|I|\geq2$ and $s(\mathcal A)=0$. Its entries are nonzero. For $r=2,3$, abbreviate the sums already occurring in the proposition by
\begin{equation}\label{eq:direct-Zr}
Z_r(\mathcal A)=
\sum_{\substack{\pi\in\Pi_r(I)\\p_B=0\ (B\in\pi)\\2\leq|B|\leq m-2(r-1)}}
\prod_{B\in\pi}|B|!.
\end{equation}
They are zero when $m<2r$. For a permutation $\sigma$ of $I$, denote by $L_+,L_-,L_0$ the numbers of labels in cycles with positive, negative, and zero sum, respectively; thus $L_++L_-+L_0=m$.

Adjoin a new label of weight $1$. Making it a singleton increases $u(\sigma)$ by one. Inserting it in a cycle increases $u(\sigma)$ by one if the cycle sum is nonnegative, and otherwise leaves $u(\sigma)$ unchanged. There are $m+1$ positions to insert $i$ into $\sigma$ in total. Hence
\begin{equation}\label{eq:direct-plus-one}
S_1(\mathcal A\sqcup\{1\})
=(m+1)S_1(\mathcal A)+m!+\sum_{\sigma\in\mathfrak S_I}(L_+(\sigma)+L_0(\sigma)).
\end{equation}
A cycle with underlying subset $B$ contributes $|B|!(m-|B|)!$ to the sum $
\sum_{\sigma} (L_+(\sigma)+L_0(\sigma))$ when $p_B\geq0$. The full subset contributes $m!$. For proper subsets, $p_B + p_{I\backslash B} = 0$. Exactly one of the two is nonnegative unless both are zero, in which case there is one additional contribution $Z_2(\mathcal A)$. Thus
\begin{equation}\label{eq:direct-nonnegative-cycles}
\sum_{\sigma\in\mathfrak S_I}(L_+(\sigma)+L_0(\sigma))
=\frac{m+1}{2}m!+Z_2(\mathcal A).
\end{equation}
Substitute \eqref{eq:direct-plus-one}--\eqref{eq:direct-nonnegative-cycles} into \eqref{eq:direct-fK}, and then use \eqref{eq:direct-gK}:
\begin{align*}
f_{K_X}^{(m+1)}(\mathcal A\sqcup\{1\})
&=(m+1)S_1(\mathcal A)-\frac{m-1}{2}m!+Z_2(\mathcal A)\\
&=-(m+1)!g_{K_X}^{(m-1)}(\mathcal A)+m!+Z_2(\mathcal A).
\end{align*}
This is exactly \eqref{eq:identity-K} with $d=m+1$.

\subsubsection{Proof of \eqref{eq:identity-K2}}
We compare the expressions of $f_{K_X^2}^{(m+2)}(\mathcal A\sqcup \{1\})$ and $g_{K_X^2}^{(m)}(\mathcal A)$. By \eqref{eq:direct-fK2}:
\[
f_{K_X^2}^{(m+2)}(\mathcal A\sqcup \{1\}) = \frac{S_2(\mathcal A\sqcup\{1,-1\})-S_1(\mathcal A\sqcup\{1,-1\})}{2}
-\frac{(m+2)!}{24}\sum_i p_i^2.
\]
We first count their insertions into a fixed permutation $\sigma$ of $I$, with $u=u(\sigma)$ and $L_+,L_-,L_0$ as above. Suppose the new label of $+1, -1$ are $\blacktriangle$,$\blacktriangledown$; $I' = I\sqcup\{\blacktriangle,\blacktriangledown\}$. We run through all permutations $\rho$ of $I'$ that reduce to $\sigma$ by deleting $\{+1,-1\}$. Notice that there is 
\[
\left.
\frac{\partial_t^2-\partial_t}{2}e^{tu}
\right|_{t=0}
=
\frac{u^2-u}{2}.
\]
So we consider $e^{tu(\sigma)}$ instead sometimes, where $t$ is formal.

There are four cases. (i) If both new labels are in cycles without old labels, they either form one cycle of weight zero or form two singleton cycles, giving $1+e^t$ in the exponential sum relative to $e^{tu}$. (ii) If exactly one new label is a singleton and the other lies in an old cycle, the contribution is
$m(1+e^t)+(e^t-1)L_0$. (iii) If both lie in the same old cycle $B$, there are $|B|(|B|+1)$ choices and the cycle weight is unchanged. (iv) Finally, if they lie in distinct old cycles, the contribution is
\[
m^2-\sum_B|B|^2
+(e^t-1)\left((L_++L_0)(L_-+L_0)
-\sum_{B:\,p_B=0}|B|^2\right)
+(e^{-t}-1)L_+L_-.
\]
The subtraction excludes the case both labels in the same zero-weight cycle. Using \(1+e^t=2+(e^t-1)\), \(\sum_B|B|=m\), and \(L_++L_-+L_0=m\), the sum of the four contributions simplifies to
\begin{equation}\label{eq:direct-two-label-count}
\begin{split}
\sum_{\rho:\,\rho|_I=\sigma}e^{t u(\rho)}
=e^{tu}\Bigg(& (m+1)(m+2)\\
&+(e^t-1)\left((m+1)(1+L_0)+L_+L_-
-\sum_{B:\,p_B=0}|B|^2\right)\\
&+(e^{-t}-1)L_+L_-\Bigg).
\end{split}
\end{equation}
Here $\rho|_I$ means deletion of labels $\{\blacktriangle, \blacktriangledown\}$, not restriction of a permutation map. The variable $t$ is only used to read the first two moments of this finite sum.

Apply $(\partial_t^2-\partial_t)/2$ at $t=0$ and sum over $\sigma$. The result is
\begin{align}
&\frac{S_2(\mathcal A\sqcup\{1,-1\})-S_1(\mathcal A\sqcup\{1,-1\})}{2}\notag\\
&\quad=\frac{(m+1)(m+2)}2 S_2(\mathcal A)
-\frac{m(m+1)}2 S_1(\mathcal A)\notag\\
&\qquad+\sum_{\sigma}u(\sigma)
\sum_{B:\,p_B=0}|B|(m+1-|B|)
+\sum_{\sigma}L_+L_-.
\label{eq:direct-two-label-moment}
\end{align}
We evaluate the last two sums separately.

For the first, fix the zero-sum cycle $B$. The complementary subset $I\backslash B$ has total sum zero and contribute $S_1(\mathcal A_{I\backslash B})$, where $\mathcal A_C=(p_i)_{i\in C}$ denotes restriction, not merging. The full cycle contributes zero. Using \eqref{eq:direct-gK} on the complement therefore gives
\begin{align}
&\sum_\sigma u(\sigma)\sum_{B:\,p_B=0}|B|(m+1-|B|)\notag\\
&\quad=\sum_{b=2}^{m-2}
\sum_{\substack{B\subseteq I\\|B|=b,\ p_B=0}}
 b!(m-b+1)S_1(\mathcal A_{I\backslash B})\notag\\
&\quad=\frac{m+2}{2}Z_2(\mathcal A)
-\sum_{\substack{\pi\in\Pi_2(I)\\p_B=0\ (B\in\pi)\\2\leq|B|\leq m-2}}
\sum_{C\in\pi}(|C|+1)!
 g_{K_X}^{(|C|-1)}(\mathcal A_C)
 \prod_{B\in\pi\backslash\{C\}}|B|!.
\label{eq:direct-marked-term}
\end{align}
For a two-block partition of sizes $b,c$, the two constant contributions have sum
$\tfrac12(b!(c+1)!+c!(b+1)!)=(m+2)b!c!/2$.

For the second sum of \eqref{eq:direct-two-label-moment}, pick one positive cycle $B$ and one negative cycle $C$. When their complement is empty, \(I=B\sqcup C\). Since \(s(\mathcal A)=0\), either $p_B>0$ and $p_C<0$, or $p_B=0$ and $p_C=0$. Hence the contribution is
\[
\sum_{\substack{\pi=\{B,C\}\in\Pi_2(I)\\ p_Bp_C<0}}
|B|!\,|C|!
=
\frac{m!(m-1)}2-Z_2(\mathcal A).
\] When it is nonempty, regard it as a third block. When $s(\mathcal A) = 0$, there are two choices of a positive and a negative block if none of the three block sums is zero, one choice if exactly one is zero, and no choice if all three are zero. The contribution is
\[
\sum_{\pi\in\Pi_3(I)}
\left(
2-\#\{B\in\pi:p_B=0\}
+\mathbf 1_{\{p_B=0\text{ for all }B\in\pi\}}
\right)
\prod_{B\in\pi}|B|!.
\]Equivalently,
\[
2\sum_{\pi\in\Pi_3(I)}\prod_{B\in\pi}|B|!
-
\sum_{\pi\in\Pi_3(I)}
\#\{B\in\pi:p_B=0\}\prod_{B\in\pi}|B|!
+
Z_3(\mathcal A).
\]
The unrestricted three-block sum is
\[
\sum_{\pi\in\Pi_3(I)}\prod_{B\in\pi}|B|!=\frac{m!(m-1)(m-2)}{12}.
\]
Also
\[
\sum_{\pi\in\Pi_3(I)}
\#\{B\in\pi:p_B=0\}\prod_{B\in\pi}|B|!
=\frac{m-2}{2}Z_2(\mathcal A).
\]
Indeed, fixing a zero block of size $b$ leaves $c=m-b$ labels whose unrestricted two-block count is $c!(c-1)/2$; pairing the two complementary zero blocks gives $(b+c-2)b!c!/2$. Consequently
\begin{equation}\label{eq:direct-opposite-count}
\sum_{\sigma\in\mathfrak S_I}L_+L_-
=\frac{(m+1)!(m-1)}6-\frac m2 Z_2(\mathcal A)+Z_3(\mathcal A).
\end{equation}
Empty ranges in these summations contribute zero.

Finally apply \eqref{eq:direct-fK2} to $\mathcal A\sqcup\{1\}$. Its polynomial term is
$-(m+2)!\sum_i p_i^2/24$. Substitution of \eqref{eq:direct-two-label-moment}, \eqref{eq:direct-marked-term}, and \eqref{eq:direct-opposite-count}, followed by \eqref{eq:direct-gK}--\eqref{eq:direct-gK2}, gives
\begin{align*}
f_{K_X^2}^{(m+2)}(\mathcal A\sqcup\{1\})
={}&-(m+2)!g_{K_X^2}^{(m)}(\mathcal A)
-(m+1)!g_{K_X}^{(m-1)}(\mathcal A)+\frac{(m+2)!}{12}\\
&+Z_2(\mathcal A)+Z_3(\mathcal A)\\
&-\sum_{\substack{\pi\in\Pi_2(I)\\p_B=0\ (B\in\pi)\\2\leq|B|\leq m-2}}
\sum_{C\in\pi}(|C|+1)!
 g_{K_X}^{(|C|-1)}(\mathcal A_C)
 \prod_{B\in\pi\backslash\{C\}}|B|!.
\end{align*}
This is \eqref{eq:identity-K2} with $d=m+2$. Now we are done with the proof of Proposition~\ref{prop:two-identities}.

\end{document}